\documentclass[a4paper,12pt]{amsart}
\usepackage{graphicx}
\usepackage{amsmath}
\usepackage{latexsym,lscape}
\usepackage{amsthm,float}
\usepackage{epic,latexsym,bezier,caption,amsbsy,psfrag,color,enumerate,amsfonts,amsmath,amscd,amssymb}
\usepackage{epsfig}
\usepackage{rotating}
\usepackage{graphics}
\usepackage[latin1]{inputenc}

\newtheorem{defi}{Definition}[section]
\newtheorem{theorem}[defi]{Theorem}
\newtheorem{lemma}[defi]{Lemma}
\newtheorem{prop}[defi]{Proposition}
\newtheorem{corollary}[defi]{Corollary}

\newtheorem{question}{Question}

\theoremstyle{definition}
\newtheorem{example}[defi]{Example}

\newtheorem{remark}[defi]{Remark}

\begin{document}


\title{Some degree and distance-based invariants of wreath products of graphs}

\author{Matteo Cavaleri}
\address{Matteo Cavaleri, Universit\`{a} degli Studi Niccol\`{o} Cusano - Via Don Carlo Gnocchi, 3 00166 Roma, Italia}
\email{matteo.cavaleri@unicusano.it}

\author{Alfredo Donno}
\address{Alfredo Donno, Universit\`{a} degli Studi Niccol\`{o} Cusano - Via Don Carlo Gnocchi, 3 00166 Roma, Italia}
\email{alfredo.donno@unicusano.it}

\begin{abstract}
The wreath product of graphs is a graph composition inspired by the notion of wreath product of groups, with interesting connections with Geometric Group Theory and Probability.
This paper is devoted to the description of some degree and distance-based invariants, of large interest in Chemical Graph Theory, for a wreath product of graphs. An explicit formula is obtained for the Zagreb indices, in terms of the Zagreb indices of the factor graphs. A detailed analysis of distances in a wreath product is performed, allowing to describe the antipodal graph and to provide a formula for the Wiener index. Finally, a formula for the Szeged index is obtained. Several explicit examples are given.

\end{abstract}

\maketitle

\begin{center}
\end{center}

\section{Introduction}

The idea of constructing new graphs starting from smaller  graphs is very natural and largely developed in literature, for its theoretical interest in several branches of Mathematics - Algebra, Combinatorics, Probability, Harmonic Analysis - but also for its practical applications in Mathematical Chemistry, as graphs are generated in a very natural way from molecules when atoms are replaced by vertices, and bonds by edges. A large number of papers in the last decades has been devoted to graph compositions, and to the investigation of their topological, combinatorial, and spectral properties.\\
\indent The most classical graph products are the Cartesian
product, the direct product, the strong product, the lexicographic product
(see \cite{imrichbook} and reference therein). More recently, the zig-zag product was
introduced \cite{annals}, in order to produce expanders of constant degree and
arbitrary size; in \cite{zigtree, ijgt}, some combinatorial and topological properties of such product, as well as connections with random walks, have been investigated. It is worth mentioning that many of these constructions play an important role in Geometric Group Theory, since it turns out
that, when applied to Cayley graphs of two finite groups, they
provide the Cayley graph of an appropriate product of these groups (see \cite{groups}, where this correspondence is shown for zig-zag products, or \cite{gc}, for the case of wreath and generalized wreath products).\\
%\indent Spectral properties of graph products have been object of an intensive study in the last decades, both for their algebraic and combinatorial interest, and for applications to Probability, Computer Science, and Mathematical Chemistry. The spectrum of a graph is defined as the spectrum of its adjacency matrix. The reader can refer, for instance, to the monograph \cite{browerspectrum} for an exhaustive treatment of spectra of graphs.\\
\indent In the context of graph compositions, an intensively studied topic of research is represented by a number of topological indices, mostly defined in terms of the degree of the vertices of the graph, or in terms of the distance between pairs of vertices. The target in this setting is to describe the considered index of the graph product in terms of the corresponding index of the factor graphs.\\
\indent Among the most studied degree-based topological indices associated with a graph there are the Zagreb indices, originally introduced by Gutman and Trinajsti\'{c} in \cite{defizagreb} in the early work of the Zagreb Mathematical Chemistry Group on the topological basis of the $\pi$-electron energy. Many reformulations and generalizations of such indices have been introduced later in literature, and several connections to many other quantities in Chemical Graph Theory have been pointed out (see also the beautiful book \cite{tri}). In the paper \cite{zagreboperations}, the Zagreb indices of some graph compositions have been investigated. \\
A fundamental distance-based index, and probably the most thoroughly examined, is the Wiener index, which is defined as the sum of the distances between all the unordered pairs of vertices of the graph. This index was introduced by  Wiener \cite{wiener} and, due to the wide
range of applications, it is nowadays largely studied. In particular, it is one of the most frequently used topological indices in
Mathematical Chemistry. For this reason, it has a strong correlation with many physical and
chemical properties of molecular compounds, whose properties do not only depend on their chemical formula,
but also on their molecular structure. A large number of papers is devoted to the study of the Wiener index of graph compositions (see, for instance, the papers \cite{wieneroperations2, wieneroperations1,wprod}).\\
\indent Another recently introduced distance-based topological index for graph is the Szeged index, introduced in \cite{defisze}, which coincides with the Wiener index when the considered graph is a tree. In the paper \cite{szeged1}, a more general condition to be satisfied in order to have the equality of Wiener and Szeged indices is described. See the paper \cite{szeged2} and references therein for some basic properties of the Szeged index, several chemical applications, and for a description of the Szeged index of a Cartesian product of graphs and some other graph compositions. See \cite{lastklavzar, szeged2} for bounds on the difference between the Szeged index and the Wiener index of a graph.

In the present paper, we focus our attention on a graph composition known in literature as the wreath product of graphs. This construction is nowadays largely studied, and different generalizations have been introduced \cite{gc, erschler}. Notice that this construction is interesting not only from an algebraic and combinatorial point of view, but also for its connection with Geometric Group Theory and Probability, via the notions of Lamplighter group and Lamplighter random walk (see, for instance, \cite{grigorchukzuk, sca1, woesssmall}). Moreover, we want to mention that in \cite{coronanoi} a matrix operation has been introduced, called wreath product of matrices, which is a matrix analogue of the wreath product of graphs, since it provides the adjacency matrix of the wreath product of two graphs, when applied to the adjacency matrices of the factors. In the paper \cite{compcomp} the wreath product of two complete graphs has been investigated: an explicit computation of the spectrum and of the Wiener index has been performed. In the paper \cite{cocktail}, jointly with F. Belardo, we have considered the wreath product of a complete graph with a Cocktail Party graph, and we have described its spectrum, together with the Zagreb indices and the Randi\'{c} index. In the paper \cite{BCD}, jointly with F. Belardo, a more general analysis involving also the study of distances and the computation of the Wiener index has been developed for the wreath product of a complete graph with a cyclic graph.

The aim of this paper is to develop a more general analysis in the more general case of the wreath product of any two graphs. After describing in Section \ref{sectionpreliminary} some basic properties of the wreath product of two graphs - connectedness, regularity, vertex-transitivity, bipartiteness (Proposition \ref{propbip}) - we investigate in Section \ref{sectionzagreb} two degree-based topological indices of such a product, namely the first and the second Zagreb indices. A formula in the very general case is provided, describing such indices in terms of the corresponding indices of the factor graphs (Theorem \ref{theoremzagreb}). In Section \ref{sectiondistances}, our attention is focused on the study of distances in a wreath product of any two graphs: an explicit formula is obtained in Theorem \ref{dista}. This analysis enables us to describe eccentricity and diameter in a wreath product (Corollary \ref{diam}), and to describe the antipodal graph (Theorem \ref{wr}). Some more properties of the antipodal graph are obtained in Corollary \ref{Hami} and Proposition \ref{proplabnew}. A large number of detailed examples is presented. The analysis of distances developed in Section \ref{sectiondistances} allows us to give, in Section \ref{sectionwiener}, a formula for the Wiener index of a wreath product of graph (see Theorem \ref{ww} and Corollary \ref{wwk}): notice that the Wiener index turns out to depend only on the Wiener index of the second factor graph, and on a vector Wiener index associated with the first factor graph (see Definition \ref{wv}). The special cases where the first factor graph is the complete graph, or the path graph, are studied in detail in Sections \ref{sub1} and \ref{sub2}. Finally, Section \ref{sectionszeged} is devoted to the study of the Szeged index. In the general case, we obtain (see Theorem \ref{szegen}) a formula in terms of the Szeged index of the second factor graph, and of a sort of generalized Szeged index of the first factor graph (see Definition \ref{dd}), defined for every pair of subsets of its vertices. Under hypothesis of edge-transitivity for the first factor graph, we obtain a strong simplification of this result (Remark \ref{simmetrie}): as an application, we provide an explicit formula for the Szeged index of a wreath product, when the first factor graph is complete (see Theorem \ref{Tkn}).

\section{Preliminary definitions}\label{sectionpreliminary}

Every graph considered in this paper will be finite, undirected, and simple, that is, loops and multiple edges are not allowed, unless explicitly specified. Such a graph will be denoted by $G = (V_G,E_G)$, where $V_G$ denotes the vertex set, and $E_G$ is the edge set consisting of unordered pairs of type $\{u,v\}$, with $u,v\in V_G$. \\
\indent If $\{u,v\}\in E_G$, we say that the vertices $u$ and $v$ are adjacent
in $G$, and we write $u\sim v$. A path of length $\ell$ in $G$ is a sequence
$u_0,u_1,\ldots, u_{\ell}$ of vertices such that $u_i\sim u_{i+1}$, for each $i=0,\ldots, \ell-1$. The path is said to be a cycle if $u_0 = u_{\ell}$. The graph $G$ is connected if, for every $u,v\in V_G$, there exists a path
$u_0,u_1,\ldots, u_{\ell}$ in $G$ such that $u_0=u$ and $u_{\ell} = v$. For a graph $G$, we will denote by $d_G(u,v)$ the \textit{geodesic distance} (or distance for short) between the vertices $u$ and $v$, that is, the length of a shortest path in $G$ joining $u$ and $v$. We put $d_G(u,v) = \infty$ if there is no path connecting $u$ and $v$. The \textit{eccentricity} of a vertex $u\in V_G$ is defined as $e_G(u)= \max_{v\in V_G}\{d_G(u,v)\}$. The \textit{diameter} of $G$ is then defined as $diam(G)= \max_{u\in V_G}\{e_G(u)\}$. In particular, the diameter of $G$ is finite if and only if $G$ is a connected graph.

The \textit{adjacency matrix} of $G=(V_G,E_G)$ is the square matrix $A=(a_{u,v})_{u,v\in V_G}$, indexed by the vertices of $G$, such that
$$a_{u,v}=\begin{cases}
1 \mbox{ if } u\sim v\\
0 \mbox{ otherwise.}
\end{cases}
$$  As the graph $G$ is undirected, $A$ is a symmetric
matrix. The \textit{degree} of a vertex $u\in V_G$ is then defined as $\deg u = \sum_{v\in V_G}a_{u,v}$. In
particular, we say that $G$ is {\it regular} of degree $r_G$, or $r_G$-regular, if $\deg u=r_G$, for each $u\in V_G$. In this case, the \textit{normalized adjacency matrix} $A'$ of $G$ is obtained as $A' = \frac{1}{r_G}A$.

%For a connected graph $G=(V_G,E_G)$, the \textit{distance matrix} $D=(d_{u,v})_{u,v\in V_G}$ of $G$ is defined to be the symmetric matrix indexed by the vertices of $G$, such that $d_{u,v} = d_G(u,v)$.

We recall now the classical definitions of vertex-transitivity, edge-transitivity, arc-transitivity (or $1$-transitivity) for a graph $G=(V_G, E_G)$ (see \cite{tutte}).
\begin{defi}
Let $G=(V_G,E_G)$ be a graph. An automorphism of $G$ is a permutation $\phi$ of $V_G$ such that $u\sim v$ if and only if $\phi(u) \sim \phi(v)$, for all $u,v \in V_G$.
\end{defi}
In the rest of the paper, we will denote by $Aut(G)$ the automorphism group of the graph $G=(V_G,E_G)$. The transitivity properties of the action of the automorphism group of a graph lead to the following definitions.

\begin{defi}
Let $G=(V_G,E_G)$ be a graph, and let $Aut(G)$ denote its automorphism group.
\begin{enumerate}
\item $G$ is vertex-transitive if, given any two vertices $u,v \in V_G$, there exists $\phi \in Aut(G)$ such that $\phi(u)=v$;
\item $G$ is edge-transitive if, given any two edges $e=\{u,v\}, f=\{u',v'\} \in E_G$, there exists $\phi \in Aut(G)$ such that $\{\phi(u),\phi(v)\}=\{u',v'\}$ (shortly $\phi(e)=f$);
\item $G$ is arc-transitive if, given any two pairs of adjacent vertices $u\sim v$ and $u'\sim v'$, there exists $\phi \in Aut(G)$ such that $\phi(u)=u'$ and $\phi(v) =v'$.
\end{enumerate}
\end{defi}

Observe that, since the definition of arc-transitivity maps one edge to another, an arc-transitive graph is also edge-transitive. Conversely, an edge-transitive graph need
not to be arc-transitive. In \cite{tutte}, Tutte proved that a connected regular graph of odd degree which is both vertex-transitive and edge-transitive is also arc-transitive. In \cite{bouwer}, an infinite family of vertex-transitive, edge-transitive, but not arc-transitive regular graphs of even degree is explicitly constructed.\\
\indent We recall now the fundamental definition of wreath product of graphs.

\begin{defi}\label{defierschler}
Let $G=(V_G, E_G)$ and $H=(V_H,E_H)$ be
two graphs. The \textit{wreath product} $G\wr H$ is the graph with vertex set $V_H^{V_G}\times V_G=
\{(f,v) | f:V_G\to V_H, \ v\in V_G\}$, where two vertices $(f,v)$
and $(f',v')$ are connected by an edge if:
\begin{enumerate}
\item ({\it edges of type I}) either $v=v'=:\overline{v}$ and $f(w)=f'(w)$ for every $w\neq \overline{v}$,
and $f(\overline{v})\sim f'(\overline{v})$ in $H$;
\item ({\it edges of type II}) or $f(w)=f'(w)$, for every $w\in
V_G$, and $v\sim v'$ in $G$.
\end{enumerate}
\end{defi}

It is a classical fact (see, for instance, \cite{woesssmall}) that the simple random walk on the graph $G\wr
H$ is the so called \textit{Lamplighter random walk}, according to the following interpretation:
suppose that at each vertex of $G$ (the \textit{base graph}) there is a lamp, whose possible states (or colors) are
represented by the vertices of $H$ (the \textit{color graph}), so that the vertex $(f,v)$ of $G\wr H$ represents the configuration of the $|V_G|$ lamps at each vertex of $G$ (for each vertex $u\in V_G$, the lamp at $u$ is in the state $f(u) \in V_H$), together with the position $v$ of a lamplighter walking on the graph $G$. At each step, the lamplighter may either go to a neighbor of the
current vertex $v$ and leave all lamps unchanged (this situation corresponds to edges of type II in $G\wr H$), or he may stay at the vertex $v \in G$, but he changes the state of the lamp which is in $v$ to a
neighbor state in $H$ (this situation corresponds to edges of type I in $G\wr H$). For this reason, the wreath product $G\wr H$ is also called the Lamplighter graph, or Lamplighter product, with base graph $G$ and color graph $H$.

It follows from Definition \ref{defierschler} that  if $|V_G|=n$ and
$|V_H|=m$, the graph $G\wr H$ has $nm^n$ vertices. It is easy to see that $G\wr H$ is connected if and only if $G$ and $H$ are connected.
Moreover, if $G$ is a regular graph of degree $r_G$ and $H$ is a regular graph of degree $r_H$, then the graph $G\wr H$ is an
$(r_G+r_H)$-regular graph.\\
\indent Notice that, in the case $|V_G|=1$, the graph $G\wr H$ is isomorphic to $H$; on the other hand, if $|V_H|=1$, then the graph $G\wr H$ is isomorphic to $G$. In the rest of the paper, we suppose that $|V_G|>1$ and $|V_H|>1$.

It is worth mentioning that the wreath product of graphs represents a graph analogue of the
classical wreath product of groups \cite{meldrum}, as it turns out that the wreath product of the Cayley graphs of two finite groups is the Cayley graph of the wreath product of the groups, with a suitable choice of the generating sets. In the paper \cite{gc}, this correspondence is proven in the more general context of generalized wreath products of graphs, inspired by the construction introduced in \cite{bayleygeneralized} for permutation groups.

In the paper \cite{coronanoi}, the wreath product of two square matrices $A$ of size $n$, and $B$ of size $m$, has been defined to be the square matrix of size $nm^n$ given by
$$
A\wr B = I_m^{\otimes^n} \otimes A + \sum_{i=1}^n I_m^{\otimes^{i-1}}\otimes B \otimes I_m^{\otimes^{n-i}}\otimes D_i,
$$
where $\otimes$ denotes the classical Kronecker product, $I_m$ is the identity matrix of size $m$, and $D_i = (d_{hk})_{h,k=1,\ldots,n}$ is the square matrix defined by
$$
d_{hk}= \left\{
          \begin{array}{ll}
            1 & \hbox{if } h=k=i \\
            0 & \hbox{otherwise.}
          \end{array}
        \right.
$$
It is also proven in \cite{coronanoi} that, if $A'_G$ is the normalized adjacency matrix of an $r_G$-regular graph $G=(V_G, E_G)$ and $A'_H$ is the normalized adjacency matrix of an $r_H$-regular graph $H=(V_H, E_H)$, then the wreath product $\left(\frac{r_G}{r_G+r_H}A'_G\right)\wr \left(\frac{r_H}{r_G+r_H}A'_H\right)$ is the normalized adjacency matrix of the graph $G\wr H$.

In Definition \ref{defierschler} the set of vertices of $G\wr H$ is given by  $V_H^{V_G}\times V_G$. Observe that, if $|V_G|=n$ and we fix an enumeration for $V_G:=\{x_1,x_2,\ldots,x_n\}$, then there is a natural bijection between the set of maps $V_H^{V_G}$ and the Cartesian product $(V_H)^n = \underbrace{V_H \times \cdots \times V_H}_{n \textrm{ times}}$, so that a vertex $w\in V_{G\wr H}$ can be written as $w=(y,x_i)$, where $y=(y_1,\ldots,y_n)\in (V_H)^n$ and $x_i\in V_G$. In the rest of the paper we will often use also the notation $w=(y_1,\ldots,y_n)x_i$ for a vertex $w$ of $G\wr H$.

It is not difficult to prove that, if $G=(V_G,E_G)$ and $H=(V_H, E_H)$ are two vertex-transitive graphs, then $G\wr H$ is a vertex-transitive graph. On the other hand, the properties of edge-transitivity and arc-transitivity are not inherited by the wreath product, the graph $K_2\wr C_4$ being a simple counterexample (see Example \ref{cappa24}).
We recall here the classical definitions of Cartesian product and direct product of graphs \cite{imrichbook}.
\begin{defi}\label{Cartesiananddirect}
Let $G=(V_G, E_G)$ and $H=(V_H,E_H)$ be two finite graphs.
\begin{itemize}
\item The \emph{Cartesian product} $G \Box H$ is the graph with vertex set $V_G\times V_H = \{(u,v): u\in V_G, v\in V_H\}$, where two vertices $(u,v)$ and $(u',v')$ are adjacent precisely if $u\sim u'$ in $G$ and $v=v'$, or if $v\sim v'$ in $H$ and $u=u'$.
\item The \emph{direct product} $G \times H$ is the graph with vertex set $V_G\times V_H = \{(u,v): u\in V_G, v\in V_H\}$, where two vertices $(u,v)$ and $(u',v')$ are adjacent precisely if $u\sim u'$ in $G$ and $v\sim v'$ in $H$.
\end{itemize}
\end{defi}
These graph products are both associative, so that we will use the notation $H^{\Box n}$ and $H^{\times n}$ for the $n$ times iterated Cartesian product and direct product of the graph $H$ with itself, respectively.

\begin{remark}\label{sottografo}
Notice that the graph $G\wr H$ can be regarded as a subgraph of
$H^{\Box n}\Box G$. This embedding also relates the geodesic distance in
$H^{\Box n}$ with the geodesic distance in $G\wr H$, as we will point out in
Theorem \ref{dista}. Moreover, if we fix $(y_1,\ldots, y_n)\in (V_H)^n$,
then the subgraph of $G\wr H$, induced by the vertices
 $\{(y_1,\ldots, y_n) x_i,\: i\in \{1,\ldots,n\}\}$, is isomorphic to $G$. On the other hand, fixing $k\in \{1,\ldots,n\}$
and $n-1$ vertices of $H$, let they be $y_1,\ldots,y_{k-1}, y_{k+1}, \ldots, y_n$, then the subgraph of $G\wr H$
induced by the vertices
 $\{(y_1,\ldots,y_{k-1},w, y_{k+1}, \ldots, y_n)x_k:\, w\in V_H \}$, is isomorphic to $H$.
\end{remark}

We have seen that regularity, connectedness, and vertex-transitivity properties are inherited by the wreath product. In the next proposition, we will focus on the bipartiteness property.

\begin{prop}\label{propbip}
Let $G=(V_G,E_G)$ and $H=(V_H,E_H)$ be two graphs. Then
$$
G\wr H \mbox{ is bipartite} \iff G \mbox{ and }  H \mbox{ are both bipartite}.
$$
\end{prop}
\begin{proof}
It is a classical fact that the Cartesian product of two graphs is bipartite if and only if each factor is bipartite, and that all subgraphs of a bipartite graph are bipartite. Then
$$
G \mbox{ and } H \mbox{ bipartite}\implies H^{\Box n}\Box G  \mbox{ bipartite}\implies G\wr H \mbox{ is bipartite},
$$
since by virtue of Remark \ref{sottografo} the graph $G\wr H$ is isomorphic to a subgraph of $H^{\Box n}\Box G$. On the other hand,
$$
G\wr H \mbox{ bipartite}\implies G \mbox{ and } H \mbox{ both bipartite},
$$
since, by virtue of Remark \ref{sottografo}, the graphs $G$ and $H$ are isomorphic to subgraphs of $G\wr H$.
\end{proof}

For each positive integer $n$, we will denote by $K_n$ the complete graph on $n$ vertices, by $C_n$ the cyclic graph on $n$ vertices, and by $P_n$ the path graph on $n$ vertices (of length $n-1$).

\begin{example}
In Fig. \ref{figK2P3}, we have represented the graph $K_2\wr P_3$. Observe that this is a bipartite graph, as both the factors are bipartite. Compare with Fig. \ref{figureantipodal1}, where the graph $K_2\wr C_3$ is depicted: this is a non bipartite graph, since the factor $C_3$ is not bipartite.
\begin{figure}[h]
\begin{center}
\includegraphics[width=0.65\textwidth]{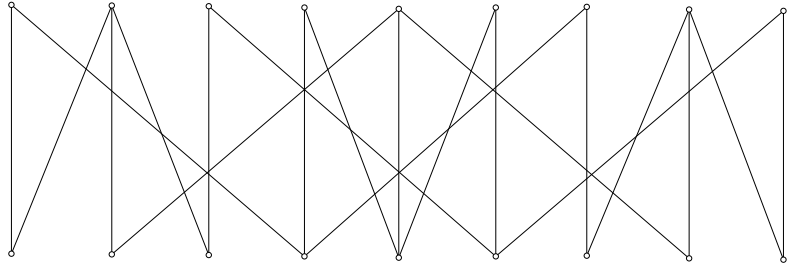}
\end{center}\caption{The graph $K_2 \wr P_3$.}  \label{figK2P3}
\end{figure}
\end{example}

\section{Zagreb indices of a wreath product}\label{sectionzagreb}

The Zagreb indices of a finite simple connected graph $G=(V_G,E_G)$ were introduced by Gutman and Trinajsti\'{c} in \cite{defizagreb} (see also the more recent survey \cite{trenta}). The {\it first Zagreb index} $M_1(G)$ is defined as
$$
M_1(G) = \sum_{v\in V_G}(\deg v)^2,
$$
whereas the {\it second Zagreb index} $M_2(G)$ is defined as
$$
M_2(G) = \sum_{u\sim v}\deg u \deg v.
$$
It follows from the definition that, if $G=(V_G,E_G)$ is a regular graph of degree $r_G$ on $n$ vertices, so that one has $|E_G|= \frac{nr_G}{2}$, the Zagreb indices are easily given by:
$$
M_1(G)  = nr_G^2 \qquad \qquad M_2(G) = \frac{nr_G^3}{2}.
$$
The aim of this section is to investigate the Zagreb indices of the wreath product $G \wr H$ of two finite simple connected graphs $G=(V_G,E_G)$ and $H=(V_H,E_H)$. We have already remarked that, if $G$ is $r_G$-regular on $n$ vertices and $H$ is $r_H$-regular on $m$ vertices, then the wreath product $G\wr H$ is a regular graph of degree $r_G+r_H$ on $nm^n$ vertices. If this is the case, then one gets for the wreath product:
\begin{eqnarray}\label{zagrebwreathregular}
M_1(G\wr H) = nm^n(r_G+r_H)^2 \qquad \qquad M_2(G\wr H) = \frac{nm^n(r_G+r_H)^3}{2}.
\end{eqnarray}

The following result gives an explicit description of the first and second Zagreb index of a wreath product $G\wr H$ in terms of the Zagreb indices of the factor graphs, in the more general case where the graphs $G$ and $H$ need not to be regular.

\begin{theorem}\label{theoremzagreb}
Let $G=(V_G,E_G)$ and $H=(V_H,E_H)$ be two finite simple connected graphs. Suppose $|V_G|=n$ and $|V_H|=m$. Then
\begin{eqnarray}\label{enunciatozagreb1}
M_1(G\wr H) = m^{n-1}(mM_1(G)+nM_1(H)+8|E_G||E_H|).
\end{eqnarray}
and
\begin{eqnarray}\label{enunciatozagreb2}
M_2(G\wr H) &=& 3m^{n-1}|E_H|M_1(G) + 2|E_G|m^{n-1}M_1(H) + m^nM_2(G)\\
 &+& nm^{n-1}M_2(H) + 4m^{n-2}|E_G||E_H|^2.\nonumber
 \end{eqnarray}
\end{theorem}
\begin{proof}
First of all, note that the degree of the vertex $v=(y_1, y_2, \ldots, y_i, \ldots, y_n)x_i$ of $G\wr H$ is $\deg x_i + \deg y_i$. By definition, we have:
\begin{eqnarray*}
M_1(G\wr H) &=& m^{n-1}\sum_{x_i\in V_G, y_i\in V_H}(\deg x_i + \deg y_i)^2\\
&=& m^{n-1}\sum_{x_i\in V_G, y_i\in V_H} ((\deg x_i)^2+(\deg y_i)^2+2\deg x_i \deg y_i) \nonumber\\
 &=& m^{n-1}(mM_1(G)+nM_1(H)+8|E_G||E_H|),\nonumber
\end{eqnarray*}
where we have used that the entries $y_j$, with $j\neq i$, can vary in $V_H$, together with the fundamental identity
\begin{eqnarray}\label{fundamentalidentity}
\sum_{v\in V_G} \deg v = 2|E_G|,
\end{eqnarray}
which holds for any graph $G$.

In order to investigate the second Zagreb index, we study the contributions to $M_2(G\wr H)$ given by edges of type I and of type II, separately. We recall that edges of type I have the form
$$
(y_1,y_2,\ldots,y_i,\ldots, y_n)x_i \sim (y_1,y_2,\ldots, \overline{y}_i, \ldots, y_n)x_i, \ \textrm{with } y_i\sim \overline{y}_i \ \textrm{in } H.
$$
Since the entries $y_j$, with $j\neq i$, can vary in $V_H$, and we do not want to consider each edge twice, we have to take into account a factor $\frac{m^{n-1}}{2}$. Then we have:
\begin{eqnarray}
M_2(G\wr H)^I &=& \frac{m^{n-1}}{2}\sum_{x_i\in V_G, y_i\in V_H, \overline{y}_i\sim y_i} (\deg y_i + \deg x_i)(\deg \overline{y}_i + \deg x_i) \nonumber\\ &=& \frac{m^{n-1}}{2}\sum_{x_i\in V_G}\sum_{y_i\in V_H, \overline{y}_i\sim y_i}(\deg y_i\deg \overline{y}_i + \deg y_i\deg x_i + \deg x_i\deg \overline{y}_i+(\deg x_i)^2)\nonumber\\
&=& \frac{m^{n-1}}{2}\left( n \sum_{y_i\in V_H, \overline{y}_i\sim y_i} \deg y_i \deg\overline{y}_i + 2 |E_G|\sum_{y_i\in V_H, \overline{y}_i\sim y_i}\deg y_i \right.\nonumber\\
&+&2\left. |E_G|\sum_{y_i\in V_H, \overline{y}_i\sim y_i}\deg \overline{y}_i + M_1(G)\sum_{y_i\in V_H, \overline{y}_i\sim y_i} 1 \right) \nonumber\\
&=& \frac{m^{n-1}}{2}\left( 2n M_2(H) + 2 |E_G|M_1(H)+2 |E_G|\sum_{y_i\in V_H} (\deg y_i)^2 +2|E_H| M_1(G) \right) \nonumber\\
&=& \frac{m^{n-1}}{2}(2n M_2(H) + 4 |E_G|M_1(H) +2|E_H| M_1(G))\nonumber,
\end{eqnarray}
where we have repeatedly used the definition of the Zagreb indices and the identity \eqref{fundamentalidentity}. Next, we recall that edges of type II have the form
$$
(y_1,y_2,\ldots,y_i,\ldots, y_n)x_i \sim (y_1,y_2,\ldots, y_j, \ldots, y_n)x_j, \ \textrm{with } x_i\sim x_j \ \textrm{in } G.
$$
Since the entries $y_h$, with $h\neq i,j$, can vary in $V_H$, and we do not want to consider each edge twice, we have to take into account a factor $\frac{m^{n-2}}{2}$. Then we have:
\begin{eqnarray}
M_2(G\wr H)^{II} &=& \frac{m^{n-2}}{2}\left( \sum_{x_i\in V_G, x_j\sim x_i} \sum_{y_i, y_j\in V_H} (\deg x_i + \deg y_i)(\deg x_j + \deg y_j)\right) \nonumber\\
&=& \frac{m^{n-2}}{2}\left(\sum_{x_i\in V_G, x_j\sim x_i} \sum_{y_i, y_j\in V_H} (\deg x_i\deg x_j + \deg x_i\deg y_j \right. \nonumber\\
&+& \left. \deg y_i\deg x_j + \deg y_i\deg y_j \right)\nonumber\\
&=& \frac{m^{n-2}}{2}\left( m^2 \sum_{x_i\in V_G, x_j\sim x_i} \deg x_i \deg x_j + 2 m|E_H|\sum_{x_i\in V_G, x_j\sim x_i}\deg x_i \right.\nonumber\\
&+& \left. 2m|E_H|\sum_{x_i\in V_G, x_j\sim x_i} \deg x_j + \sum_{x_i\in V_G, x_j\sim x_i} 4|E_H|^2 \right) \nonumber
\end{eqnarray}
\begin{eqnarray*}
&=& \frac{m^{n-2}}{2}\left( 2m^2 M_2(G) + 4 m|E_H|\sum_{x_i\in V_G}(\deg x_i)^2 +8|E_G||E_H|^2\right) \nonumber\\
&=& \frac{m^{n-2}}{2}(2m^2 M_2(G) + 4m |E_H|M_1(G) +8|E_G||E_H|^2)\nonumber,
\end{eqnarray*}
where we have used again the definition of the Zagreb indices and the identity \eqref{fundamentalidentity}.
Then, if we glue together the two contributions, we have
$$
M_2(G\wr H) = M_2(G\wr H)^{I} + M_2(G\wr H)^{II}
$$
and we obtain the assert.
\end{proof}

\begin{remark}\rm
Note that, in the case of an $r_G$-regular graph $G=(V_G,E_G)$ and an $r_H$-regular graph $H=(V_H,E_H)$, the formulas \eqref{enunciatozagreb1} and \eqref{enunciatozagreb2} coincide with the formulas given in \eqref{zagrebwreathregular}.
%Moreover, the formulas \eqref{enunciatozagreb1} and \eqref{enunciatozagreb2} recover the results obtained in \cite{BCD} and \cite{cocktail}, where an explicit computation of the Zagreb indices of the graphs $K_n\wr C_m$ and $K_n\wr CP_{2m}$, respectively, has been given, where $CP_{2m}$ is the Cocktail Party graph on $2m$ vertices.
\end{remark}

%\subsection{The special cases $G=K_n$ and $G=P_n$}

%For the complete graph $K_n$, we have $|V_{K_n}| =n$ and $|E_{K_n}| = \frac{n(n-1)}{2}$. In particular, $K_n$ is an $(n-1)$-regular graph, and the following identities hold:
%$$
%M_1(K_n) = n(n-1)^2 \qquad M_2(K_n) = \frac{n(n-1)^3}{2}.
%$$
%Therefore, if we apply Equations \eqref{enunciatozagreb1} and \eqref{enunciatozagreb2} to the wreath product $K_n \wr H$, we get by direct computation:
%\begin{eqnarray*}
%M_1(K_n\wr H) = nm^n(n-1)^2 + nm^{n-1}M_1(H) +4n(n-1)m^{n-1}|E_H|
%\end{eqnarray*}
%and
%\begin{eqnarray*}
%M_2(K_n\wr H) &=& 3nm^{n-1}(n-1)^2|E_H| + n(n-1)m^{n-1}M_1(H) + \frac{nm^n(n-1)^3}{2}\\
%&+&  nm^{n-1}M_2(H) + 2nm^{n-2}(n-1)|E_H|^2.
%\end{eqnarray*}
%In the papers \cite{BCD} and \cite{cocktail}, an explicit computation of the Zagreb indices of the graphs $K_n\wr C_m$ and $K_n\wr CP_{2m}$, respectively, has been given, where $CP_{2m}$ is the Cocktail Party graph on $2m$ vertices.

%Similarly, for the path graph $P_n$, we have $|V_{P_n}| =n$ and $|E_{P_n}| = n-1$. In particular, $P_n$ has $2$ vertices of degree $1$, and $n-2$ vertices of degree $2$. Therefore we have, for $n\geq 3$:
%$$
%M_1(P_n) = 4n-6 \qquad M_2(P_n) = 4n-8.
%$$
%Now, if we apply Equations \eqref{enunciatozagreb1} and \eqref{enunciatozagreb2} to the wreath product $P_n \wr H$, we get by direct computation:
%\begin{eqnarray*}
%M_1(P_n\wr H) = m^n(4n-6) + nm^{n-1}M_1(H) +8(n-1)m^{n-1}|E_H|
%\end{eqnarray*}
%and
%\begin{eqnarray*}
%M_2(P_n\wr H) &=& 3m^{n-1}(4n-6)|E_H| + 2(n-1)m^{n-1}M_1(H) + m^n(4n-8)\\
%&+&  nm^{n-1}M_2(H) + 4m^{n-2}(n-1)|E_H|^2.
%\end{eqnarray*}

\section{Distances in a wreath product}\label{sectiondistances}

The geodesic distance in a wreath product has been especially studied for Cayley graphs by the investigation of the Word length for wreath products of finite and infinite groups \cite{burillo, tara,olga}. In any  approach it appears an NP-hard problem: the Traveling Salesman Problem (TSP). It is one of the most intensively studied problem in optimization.
The common strategy of these works to avoid the complexity of the TSP is to consider analog easier problems that are, in some sense, approximations for the TSP: it works because most invariants in Geometric Group Theory are defined up to quasi-isometry of groups (and therefore Cayley graphs). \\
Our approach here is purely combinatorial: we push it as far as possible, in the general case, and then we apply the general results to some special classes of graphs. This yields an explicit description of the antipodal graph of a wreath product (Section \ref{santi}) and of some distance-based invariants like the Wiener index (Section \ref{sectionwiener}) and the Szeged index (Section \ref{sectionszeged}).

We start our study about distances in a wreath product by introducing the following definition.
\begin{defi}\label{roa}
Let $G=(V_G,E_G)$ be a graph and let $A\subseteq V_G$. We define a map $\rho_A$ on $V_G \times V_G$ such that, for any $u,v\in V_G$, the number $\rho_A(u,v)$ is the length of a shortest path starting from $u$, arriving at $v$, visiting all vertices of $A$. In the case the subset $A$ is the whole $V_G$, we write $d_{Ha}:=\rho_{V_G}$.
\end{defi}

Notice that, for each $u,v\in V_G$, the value $\rho_A(u,v)$ is a nonnegative integer, which represents the length of a solution for an instance of the TSP, where repetitions of vertices are allowed. The map $\rho_A$ is symmetric by definition, that is, $\rho_A(u,v)=\rho_A(v,u)$ for any $u,v\in V_G$. Moreover $\rho_{\emptyset}(u,v)=d_G(u,v)$, that is, if $A=\emptyset$, then $\rho_A$ coincides with the usual geodesic distance. The following lemma holds.

\begin{lemma}\label{monotonia}
For any $A,B\subseteq V_G$, for any $u,v,w\in V_G$,
$$\rho_{A\cup B}(u,v)\leq \rho_{A}(u,w)+ \rho_{B}(w,v).$$
\end{lemma}
\begin{proof}
The union of a path from $u$ to $w$ visiting $A$, and of a path from $w$ to $v$ visiting $B$, is a path from $u$ to $v$ visiting $A\cup B$.
\end{proof}

In particular, we have:
\begin{itemize}
\item $\rho_{A}(u,v)\leq \rho_{A}(u,w)+ \rho_{A}(w,v)\;\;\;\;$ (Triangular inequality for $\rho_A$);
\item $B\subseteq A \implies \rho_B(u,v)\leq \rho_A(u,v)\;\;\;\;$ (Monotonicity of $\rho_A$);
\item $\rho_{A}(u,v)\leq \rho_{A}(u,w)+ d_G(w,v)$.
\end{itemize}
The following is another characterization of $\rho_A$. We denote by $Sym(n)$ the symmetric group on $n$ elements.

\begin{prop}\label{permu}
Let $\emptyset \neq A\subseteq V_G$, with $A=\{a_1,\ldots, a_k\}$. Then
$$
\rho_A(u,v)= \min_{\sigma\in Sym(k)} \left\{ d_G(u,a_{\sigma(1)})+\sum_{i=1}^{k-1} {d_G(a_{\sigma(i)}, a_{\sigma(i+1)})} + d_G(a_{\sigma(k)},v) \right\}.
$$
\end{prop}
\begin{proof}
For any path from $u$ to $v$ visiting the vertices $a_1,\ldots, a_k$ we define a permutation $\sigma\in Sym(k)$ in this way: $\sigma^{-1}(i)=j$ if $a_i$ is the $j$-th
   vertex of $A$ visited for the first time in
   the path. The length of the path cannot be less than
    $d_G(u,a_{\sigma(1)})+\sum_{i=1}^{k-1}
     {d_G(a_{\sigma(i)},
     a_{\sigma(i+1)})}
      + d_G(a_{\sigma(k)},v)
      $. Thus
$$
\rho_A(u,v)\geq \min_{\sigma\in Sym(k)} \{ d_G(u,a_{\sigma(1)})+\sum_{i=1}^{k-1} {d_G(a_{\sigma(i)}, a_{\sigma(i+1)})} + d_G(a_{\sigma(k)},v) \}.
$$
Conversely, for any permutation $\sigma\in Sym(k)$, we construct a path that is the union of a shortest path from $u$ to $a_{\sigma(1)}$, with a shortest path from $a_{\sigma(1)}$ to $a_{\sigma(2)}$, and so on. Since the length of this path is exactly
$d_G(u,a_{\sigma(1)})+\sum_{i=1}^{k-1} {d_G(a_{\sigma(i)}, a_{\sigma(i+1)})} + d_G(a_{\sigma(k)},v)$, we have proven also the inverse inequality, and the proof is completed.
Notice that in the particular case $k=1$, with $A=\{a\}$, we have $\rho_A(u,v)=d_G(u,a)+d_G(a,v)$.
\end{proof}

As a corollary, we have the following bounds for $\rho_A$ in terms of the cardinality of $A$ and of the diameter of the graph $G$. More precisely, we have:
\begin{equation}\label{boundH}
\max\{d_G(u,v),|A|-1\} \leq \rho_{A}(u,v)\leq diam(G) (|A|+1).
\end{equation}
In particular, the lower bound is reached when $A$ only contains vertices that appear in a shortest path between $u$ and $v$; the upper bound in reached, for example, in the complete graph, as we will see in Example \ref{K_n}.
%\begin{remark}\label{punto}
 %For every $a\in A$ and every $u,v\in V_G$ we have $\rho_A(u,v)\geq d_G(u,a)+d_G(a,v)$. When $|A|=1$ the equality holds.
%\end{remark}
The geodesic distance $d_G$ is clearly invariant under automorphism; on the other hand, for the map $\rho_A$ the situation is different, as the following proposition shows.
\begin{prop}\label{inva}
For any $\phi\in Aut(G)$, any $A\subseteq V_G$, and for all vertices $u,v\in V_G$,
$$\rho_A(u,v)=\rho_{\phi(A)}(\phi(u),\phi(v)).$$
\end{prop}
\begin{proof}
Let $u=x_0\sim x_1 \sim x_2 \sim \cdots \sim x_{k-1}\sim x_k=v$ be a minimal path from $u$ to $v$ visiting $A$, so that $\rho_A(u,v)=k$.
Then the path $\phi(x_0)\sim \phi(x_1)\sim\cdots\sim \phi(x_k)$ is  a path from   $\phi(u)$ to  $\phi(v)$ visiting  $ \phi(A)$, so that ${\rho_A(u,v)\geq\rho_{\phi(A)}(\phi(u),\phi(v))}$. Now by applying the inverse automorphism $\phi^{-1}$ we get the inverse inequality, and the claim follows.
\end{proof}

We recall the following classical definitions.
\begin{defi}
A graph $G=(V_G,E_G)$ is \emph{Hamiltonian} if it contains a \emph{Hamiltonian cycle}, that is, a cycle that visits each vertex of $G$ exactly once. A graph $G=(V_G,E_G)$ is \emph{Hamilton-connected} if, for every $u,v \in V_G$,  there exists a \emph{Hamiltonian path} from $u$ to $v$, that is, a path from $u$ to $v$ visiting each vertex of $G$ exactly once.
\end{defi}
Note that, if $G$ is Hamilton-connected, then it is also a Hamiltonian graph. The viceversa is false, the cyclic graph being a simple counterexample (for a detailed analysis see Example \ref{C_n}). Moreover, all bipartite graphs are not Hamilton-connected; on the other hand, the complete graph is Hamilton-connected.

\begin{remark}\label{hamidis}
The property of being Hamiltonian or Hamilton-connected is detectable by means of the map $d_{Ha}$. More precisely,
if $G=(V_G,E_G)$ is a connected graph with $|V_G|=n$, we have:
\begin{equation}\begin{split} \label{HC}
\exists u\in V_{G} \mbox{ (equivalently $\forall u\in V_{G}$)}: \;d_{Ha}(u,u)=n&\iff G \mbox{ is Hamiltonian},
\end{split}\end{equation}
\begin{equation}\begin{split} \label{HC1}
\forall u,v\in V_{G},\;d_{Ha}(u,v)=
\begin{cases}
n-1 \; &\mbox{if } u\neq v \\
n \; &\mbox{if } u=v.
\end{cases} &\iff G \mbox{ is Hamilton-connected}.
\end{split}\end{equation}
\end{remark}

In our setting the map $d_{Ha}$ plays the role of a distance in some sense: we can reasonably define a notion of Hamiltonian eccentricity and Hamiltonian diameter replacing the geodesic distance $d_G$ with the map $d_{Ha}$.
\begin{defi}
For any $u\in V_G$, the \emph{Hamiltonian eccentricity} of the vertex $u$ is $e_{G,Ha}(u):=\max\{d_{Ha}(u,v),\; v\in V_G\}$. Similarly, the \emph{Hamiltonian diameter} of the graph $G$ is $diam_{Ha}(G):=\max\{e_{G,Ha}(u),\; u\in V_G\}$.
\end{defi}

In particular, if a finite simple connected graph $G=(V_G,E_G)$, with $|V_G|=n$, is Hamilton-connected, then $diam_{Ha}(G)=n$ and all the shortest paths starting and ending at the same vertex, visiting any other vertex, realize the Hamiltonian diameter.

\begin{example}\label{K_n}
Let $G=K_n$ be the complete graph on $n$ vertices ($n\geq 2$).\\
If $A\subseteq V_{K_n}$, $A\neq\emptyset$, and $u,v\in V_{K_n}$, then it is easy to check that:
\begin{equation}\label{rocompleto}
\rho_A(u,v)=\begin{cases}
|A|+1\; &\mbox{ if } u,v\notin A\\
|A|\;  &\mbox{ if } (u\notin A, v\in A) \mbox{ or }(u\in A, v\notin A)\\
|A|-1\;  &\mbox{ if } u,v\in A , u\neq v\\
|A|\;  &\mbox{ if } u=v\in A, |A|>1\\
0 \;  &\mbox{ if } u=v\in A, |A|=1.\\
\end{cases}
\end{equation}
In particular, the graph $K_n$ is Hamilton-connected, and Equations \eqref{HC} and \eqref{HC1} are verified. Moreover
$$
e_{K_n, Ha}(u)=diam_{Ha}(K_n)=n, \qquad \textrm{ for each } u\in V_{K_n}.
$$
In Fig. \ref{figurecomplete6}, we have considered the complete graph $K_6$, with a fixed labelling of the vertex set, and the subset $A=\{1,2\}$ of $V_{K_6}$. In this case, regarding $\rho_A$  as the $n\times n$ symmetric matrix such that $({\rho_A})_{i,j}=\rho_A(i,j)$, for any $i,j\in \{1,\ldots,6\}$, we get:
$$
\rho_A=\left(
                \begin{array}{cccccc}
                  2 & 1 & 2& 2&2 & 2\\
                 1 &2&  2& 2 &2& 2\\
                  2 & 2& 3& 3 &3 &3 \\
                  2 & 2& 3& 3 &3 &3 \\
                   2 &2&  3& 3 &3& 3\\
                  2 & 2& 3& 3&3 &3 \\
                \end{array}
              \right).
$$
\begin{figure}[h]
\begin{center}
\psfrag{1}{$1$}\psfrag{2}{$2$}\psfrag{3}{$3$}\psfrag{4}{$4$}
\psfrag{5}{$5$}\psfrag{6}{$6$}
\includegraphics[width=0.3\textwidth]{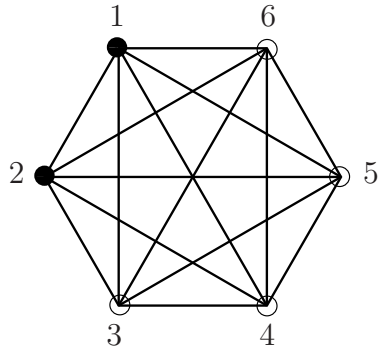}
\end{center}\caption{The complete graph $K_6$ with $A=\{1,2\}$.} \label{figurecomplete6}
\end{figure}
\end{example}

\begin{example}\label{P_n}
Let $G=P_n$ be the path graph on $n$ vertices ($n\geq 2$). Set $V_{P_n}=\{1,2,\ldots, n\}$, with $1\sim 2\sim\cdots\sim n-1\sim n$. With this vertex labeling, the geodesic distance in $P_n$  is given by $d_{P_n}(u,v)=|u-v|$, where $|x|$ denotes the absolute value of the real number $x$.

Notice that, for $A\subseteq V_{P_n}$, $A\neq\emptyset$, the positive integers $\min A$ and $\max A$ are vertices of $P_n$.
In order to describe $\rho_A$, we just have to take into account that
$$
\rho_A=\rho_{\{\min A,\max A\}},
$$
for each $\emptyset\neq A\subseteq V_{P_n}$. This is true because a path visits the set ${\{\min A,\max A\}}$ if and only if it visits the whole set $A$.
Thus, by virtue of Proposition \ref{permu},
$$
\rho_A(u,v)=\min\{|u-\min A|+|v-\max A|, |u-\max A|+ |v-\min A|\}+(\max A-\min A).
$$
If we consider all possible cases, we have, for $u\leq v$ ($\rho_A$ is symmetric):
\begin{equation}\label{patpat}
\rho_A(u,v)=\begin{cases}
v-u\; &\mbox{ if } u< \min A, v> \max A\\
2\max A-(u+v)\;  &\mbox{ if } u< \min A,v\leq \max A\\
-2\min A+(u+v)\;  &\mbox{ if } u\geq \min A, v> \max A\\
2 (\max A-\min A) -(v-u)\;  &\mbox{ if }  \min A\leq u,v\leq \max A.\\
\end{cases}
\end{equation}
In particular, for $A=V_{P_n}$, we get $d_{Ha}(u,v)=2(n-1)-d_{P_n}(u,v)$, so that
$$
e_{P_n, Ha}(u)=diam_{Ha}(P_n)=d_{Ha}(u,u)=2n-2, \qquad \textrm{ for each } u\in V_{P_n}.
$$
This implies that  even if $P_n$ is not Hamilton-connected, the paths whose lengths realize the Hamiltonian eccentricity and Hamiltonian diameter are those starting and ending at the same vertex. In Fig. \ref{figurepath9}, we have represented the path graph $P_9$ and the vertex subset $A=\{3,5,6\}$.
We have $\min A= 3$ and $\max A=6$. By applying \eqref{patpat}, or by explicitly computing the lengths of the paths, we obtain:
$$
\rho_A=\left(
                \begin{array}{cc|cccc|ccc}
                  10 &9& 8&7&6 & 5&6&7&8\\
                  9 &8&  7& 6 &5& 4&5&6&7\\
                      \hline
                  8 & 7& 6& 5 &4&3&4&5&6 \\
                  7 & 6& 5& 6 &5&4 &5&6&7\\
                  6 &5&  4& 5 &6& 5&6&7&8\\
                  5 &4&  3& 4 &5& 6&7&8&9\\
                      \hline
                  6 &5&  4& 5 &6& 7&8&9&10\\
                  7&6&  5& 6 &7& 8&9&10&11\\
                  8 & 7& 6& 7&8&9 &10&11&12\\
                \end{array}
              \right).
$$
\begin{figure}[h]
\begin{center}
\psfrag{1}{$1$}\psfrag{2}{$2$}\psfrag{3}{$3$}\psfrag{4}{$4$}
\psfrag{5}{$5$}\psfrag{6}{$6$} \psfrag{7}{$7$}\psfrag{8}{$8$}\psfrag{9}{$9$}
\includegraphics[width=0.4\textwidth]{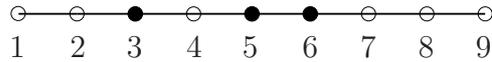}
\end{center}\caption{The path graph $P_9$ with $A=\{3,5,6\}$.} \label{figurepath9}
\end{figure}
\end{example}

\begin{example}\label{C_n}
Let $G=C_n$ be a cycle, with $n>2$. Then:
\begin{equation}\label{distaci}
d_{Ha}(u,v)=
\begin{cases}
n-2+ d_{C_n}(u,v) \; &\mbox{ if } u\neq v \\
n \;&\mbox{ if } u=v.
\end{cases}
\end{equation}
\begin{proof}
As the graph $C_n$ is Hamiltonian, we have  $d_{Ha}(u,u)=n$, for every $u\in V_{C_n}$ (see Remark \ref{hamidis}). When $d_{C_n}(u,v)=1$, so that $u\sim v$ in $C_n$, then it is clear that there exists a path of length $n-1$ from $u$ to $v$ visiting any other vertex of $C_n$. There is no shorter path with this property, since the number of vertices to be visited is $n$ (see lower bound in \eqref{boundH}).

When $d_{C_n}(u,v)>1$, the graph obtained by removing $u$ and $v$ has $2$ connected components.
Therefore a path from $u$ to $v$ visiting any other vertex has to visit at least twice at least one between $u$ and $v$. Let
$$
u=x_0\sim x_1\sim x_2\sim\cdots\sim x_{\ell}=v
$$
be such a path. Suppose that $u$ is visited twice (the other case being analog); then the subpath from $x_1$ to $v$ is a path of length $\ell-1$ visiting any vertex of $C_n$.  Then, either $d_{C_n}(x_1,v)=1$ and we have done, or we can iterate the same argument: a path from $x_1$ to $v$ visiting any vertex of $C_n$ has to visit at least twice at least one among $x_1$ and $v$ and so on. More precisely, by repeating the argument $d_{C_n}(u,v)-1$ times, we obtain a path of length  $\ell- (d_{C_n}(u,v)-1)$ visiting all vertices, so that $\ell- (d_{C_n}(u,v)-1)\geq n-1$. This implies $d_{Ha}(u,v)\geq n-2+ d_{C_n}(u,v)$. Finally, we prove the thesis by noticing that there exists a good path of length exactly $n-2+ d_{C_n}(u,v)$. For instance, let $u^*$ denote the neighbour of $u$ which is closest to $v$: then the union of the shortest path from $u$ to $u^*$ visiting any other vertex and the minimal path from $u^*$ to $v$ is a path from $u$ to $v$ visiting any other vertex of the desired length.
\end{proof}
As a consequence observe that, when $n>3$ (and therefore we are not in the complete case), in light of Remark \ref{hamidis}, the graph $C_n$ is not Hamilton-connected. More precisely, we have
$$
e_{C_n, Ha}(u)= n-2+e_{C_n}(u) \qquad \qquad diam_{Ha}(C_n)=n-2+diam(C_n).
$$
\end{example}
\vspace{0.5cm}

Recall that if we fix an ordering for the vertices of $G$, say $V_G=\{x_1,\ldots,x_n\}$, a vertex $u$ of $G\wr H$ can be written as $u=(y, x_i)$, with $y\in (V_H)^n$ and $i\in \{1,\ldots, n\}$. The $n$-tuple $y$ in $(V_H)^n$ can be regarded as a vertex of the $n$-th Cartesian power of $H$ (see Definition \ref{Cartesiananddirect}). It is well known that the distance in the Cartesian product is nothing but the sum of the distances computed coordinatewise. In other words, given two vertices $y=(y_1,\ldots, y_n)$ and $ y'=(y'_1,\ldots, y'_n)$ in $V_{{H^{\Box n}}}$, with $y_i, y'_i \in V_H$ for each $i=1,\ldots,n$, one has:
$$
d_{{H^{\Box n}}}(y,y')=\sum_{i=1}^n d_H(y_i,y'_i).
$$
Two vertices $y$ and $y'$ of $V_{H^{\Box n}}$, thanks to the ordering of $V_G$, define a subset of $V_G$ denoted by
$$
\delta(y,y'):= \{x_i\in V_G: y_i\neq y'_i \}.
$$
\begin{theorem}\label{dista}
Let $G=(V_G,E_G)$ and $H=(V_H,E_H)$ be two connected graphs. Suppose $|V_G|=n$ and $|V_H|=m$. For any vertices $u=(y_1,\ldots,y_n)x$, $v=(y_1',\ldots, y_n')x'\in G\wr H$, we have:
$$
d_{G\wr H}(u,v)=d_{H^{\Box n}}(y,y')+\rho_{\delta(y,y')}(x,x')= \sum_{i=1}^n d_H(y_i,y'_i) + \rho_{\delta(y,y')}(x,x'),
$$
where $y=(y_1,\ldots,y_n)$ and $y'=(y_1',\ldots,y_n')$.
\end{theorem}
\begin{proof}
It is easy to see that there exists a path of length $d_{H^{\Box n}}(y,y')+\rho_{\delta(y,y')}(x,x')= \sum_{i=1}^n d_H(y_i,y'_i) + \rho_{\delta(y,y')}(x,x')$ connecting $u$ and $v$. On the other hand, any path from $u$ to $v$ can be decomposed into the disjoint union of a path with only edges of type I and a path with only edges of type II. The first is actually a path in ${H^{\Box n}}$ from $y$ to $y'$, and then its length has to be not less than $d_{{H^{\Box n}}}(y,y')$. The second is a path from $x$ to $x'$ in $G$ that must visit every vertex in $\delta(y,y')$. By Definition \ref{roa}, its length has to be not less than $\rho_{\delta(y,y')}(x,x')$ and the proof is done.
\end{proof}

\begin{remark}
In the works \cite{burillo,tara} an analog formula is given for the word length of an element in a wreath product of finitely generated groups, via canonical form.
\end{remark}
In the Lamplighter interpretation, in order to connect the two vertices $u$ and $v$, the lamplighter has to visit each position where the colors of the lamp differ (the subset $\delta(y,y')$), starting from $x$, arriving at $x'$, switching the colors of the lamps at each of this position.

\begin{corollary}\label{diam}
For any $u=(y_1,\ldots, y_n)x \in V_{G\wr H} $ we have:
$$
e_{G\wr H}(u)= \sum_{i=1}^n e_H(y_i)+ e_{G,Ha}(x).
$$
In particular: $$diam(G\wr H)= n \, diam(H)+ diam_{Ha}(G).$$
\end{corollary}
\begin{proof}
For every $i=1,\ldots, n,$ let $y'_i$  be a vertex of $V_H$ satisfying  $d_H(y_i,y'_i)=e_H(y_i)$. Set $y'=(y_1', \ldots, y_n')$.
It is clear that $\delta(y,y')= V_G$ and then, if $x'\in V_G$ is a vertex such that $d_{Ha}(x,x')= e_{G,Ha}(x)$, we have
$d_{G\wr H}((y_1,\ldots,y_n)x,(y_1',\ldots, y_n')x')=  \sum_{i=1}^n e_H(y_i)+ e_{G,Ha}(x)$.\\ Since $d_{Ha}\geq \rho_A$ for any $A\subseteq V_G$ (monotonicity of $\rho_A$),
there is no vertex in $V_{G\wr H}$ with distance from $u$ more than $d_{G\wr H}((y_1',\ldots, y_n')x',u)$.
\end{proof}

\subsection{Antipodal Graph}\label{santi}

The notion of antipodal graph of a given graph was introduced in \cite{antipodal2}. In \cite{antipodal1}, the antipodal of the four classical graph products, namely the Cartesian, the direct, the strong and the lexicographic product, is investigated and described in terms of the antipodal graphs of the factors. The aim of this section is to study the antipodal graph of the wreath product of two given graphs, by using the analysis performed in the first part of this section.
\begin{defi}
Let $G=(V_G, E_G)$ be a connected graph. The \emph{antipodal graph} of $G$, denoted by $A(G)$, is the graph whose vertex set coincides with $V_G$, where two vertices $u$ and $v$ are adjacent if $d_G(u,v) = diam (G)$.
\end{defi}

\begin{example}\label{exclas} The diameter of the complete graph $K_n$ is $1$ and then $A(K_n)=K_n.$ \\
The diameter of the cycle $C_n$ is $\lfloor  \frac{n}{2} \rfloor$, then the antipodal graph $A(C_n)$ is a cycle itself if $n$ is odd, and it is the disjoint union of the edges between opposite vertices if $n$ is even. In formulas, we have:
$$
A(C_n)=\left\{
         \begin{array}{ll}
           C_n & n \hbox{ odd} \\
           \frac{n}{2}K_2 & n \hbox{ even.}
         \end{array}
       \right.
$$
The diameter of the path graph $P_n$ is $n-1$ and it is reached only by the pair $\{1,n\}$: the antipodal graph $A(P_n)=K_2\sqcup \{\mbox{$n-2$ isolated vertices}\}$, where the symbol $\sqcup$ denotes the disjoint union.
\end{example}

Observe that, by Corollary \ref{diam}, the diameter of $G\wr H$ depends on the diameter of $H$ and on the Hamiltonian diameter of $G$. Therefore, in order to describe the antipodal graph of $G\wr H$, we introduce the definition of \emph{Hamiltonian antipodal graph}.

\begin{defi}
The \emph{Hamiltonian antipodal graph} of a graph $G=(V_G,E_G)$, denoted by $A_{Ha}(G)$, is the graph whose vertex set coincides with $V_G$, where two vertices $u$ and $v$ are adjacent if $d_{Ha}(u,v) = diam_{Ha} (G)$.
\end{defi}

Notice that, since $d_{Ha}(u,u)\neq 0$,the Hamiltonian antipodal graph is not in general a simple graph, and loops may appear. Let us denote by $O_n$ the graph on $n$ vertices whose edge set consists exactly of $n$ loops, one at each vertex. In particular, this graph consists of $n$ connected component. Moreover, for a given simple graph $G=(V_G,E_G)$, we define the (non simple) graph $G^\circ=(V^\circ, E^\circ)$ by putting $V_{G^\circ}=V_G$ and $E_{G^\circ}=E_G\cup\{\{u,u\}:\, u\in V_G\}$, that is, $G^\circ$ is the graph $G$ with in addition a loop at each vertex. Then the following is a reformulation of Remark \ref{hamidis}.
\begin{prop}\label{contra}
Let $G=(V_G,E_G)$ be a graph, with $|V_G|=n$. Then $G$ is Hamilton-connected $\iff$ $A_{Ha}(G)=O_n$ and $diam_{Ha}(G)=n$.
\end{prop}

\begin{example}\label{esempi}
Looking at Examples \ref{K_n}, \ref{P_n}, \ref{C_n}, we have $A_{Ha}(K_n)=O_n$ and $A_{Ha}(P_n)=O_n$: the graph $P_n$ is not Hamilton-connected for $n>2$, as $diam_{Ha}(P_n)=2n-2\neq n$ and therefore there is no contradiction with Proposition \ref{contra}. In the case of the cycle $C_n$, we have:
\begin{equation*}
A_{Ha}(C_n)=
\begin{cases}
A(C_n)\; &n\geq 6\\
A(C_n)^\circ \; &n=4,5.
\end{cases}
\end{equation*}
\begin{proof}
Since $C_3=K_3$, we can assume $n\geq 4$. Recall that we have (see Example \ref{C_n})
$$
diam_{Ha}(C_n)=n-2+diam(C_n)=n-2+ \left\lfloor  \frac{n}{2} \right\rfloor.
$$
Then in the case $n=4$ or $n=5$ we have  $diam_{Ha}(C_n)=n$. By Equation \eqref{distaci} we have $d_{Ha}(u,v)=n$ if and only if $d_{C_n}(u,v)=2=diam(C_n)$ or $u=v$.
In the case $n\geq 6$, we have $diam_{Ha}(C_n)>n$ and then $d_{Ha}(u,v)=diam_{Ha}(C_n)$ if and only if $d_{C_n}(u,v)=diam(C_n)$.
\end{proof}
\end{example}

%\begin{example}
%Let $K_{n,m}$ be the complete bipartite graph: it can be easily verified that
%\begin{equation*}
%A_{Ha}(K_{n,m})=
%\begin{cases}
%K_n^\circ \sqcup K_n^\circ\; &n=m\\
%K_n^\circ\sqcup  O_m\; &n<m.\\
%\end{cases}
%\end{equation*}
%Where the symbol $\sqcup$ represent the disjoint union. Notice that  the (classical) antipodal graph $A(K_{n,m})=K_n\sqcup K_m$ is the disjoint union of the complete graphs associated with its two parts.
%\end{example}

%\begin{example}
%Let $Q_n$ be the $n$-dimensional Hypercube graph (that is the $n$-fold Cartesian power of $K_2$). It is well known that $H_n$ admits a 2-coloring for the vertices (odd or even distance from a fixed vertex) and there exists a Hamiltonian path between $u, v\in V_{Q_n}$ if and only if $u$ and $v$ have different colors. Then the Hamiltonian diameter is $|V_{Q_n}|=2^n$ and it is realized by all pairs of vertices of the same color, that is:
%$$A_{Ha}(Q_n)= K_{2^{n-1}}^\circ\sqcup K_{2^{n-1}}^\circ .$$
%\end{example}
In \cite{antipodal1} it is proved that the antipodal of the Cartesian product of connected graphs is the direct product of the antipodal graphs of the factors. The following theorem is the analogue for the wreath product: the novelty consists in the fact that, in the wreath product case, the Hamiltonian antipodal of the first factor graph must be considered.

\begin{theorem}\label{wr}
Let $G$ and $H$ be two connected graphs, with $|V_G|=n$. We have:
$$
A(G \wr H)= A(H)^{\times n}\times A_{Ha}(G).
$$
\end{theorem}
\begin{proof}
According to Theorem \ref{dista} and Corollary \ref{diam}, for any pair of vertices $u=(y_1,\ldots,y_n)x$ and $v=(y_1',\ldots, y_n')x'\in G\wr H$, we have that
$d_{G\wr H}(u,v)=diam(G\wr H)$ if and only if $d_H(y_i, y_i')= diam(H)$ for each $i=1,\ldots, n$, and $\rho_{\delta(y,y')}(x,x')=diam_{Ha}(G)$. This is possible if and only if, for each $i=1,\ldots, n$, we have $y_i\sim y_i' $ in $A(H)$
and $x\sim x'$ in $A_{Ha}(G)$, that is, if and only if $(y_1,\ldots,y_n)x \sim (y_1',\ldots, y_n')x'$ in $A(H)^{\times n}\times A_H(G)$.
\end{proof}

\begin{corollary}\label{Hami}
If $G$ is also Hamilton-connected, we have
$$
A(G\wr H)=A(H)^{\times n}\times O_n \qquad \textrm{ and } \qquad diam(G\wr H) = n(diam (H) +1).
$$
\end{corollary}
\begin{proof}
It directly follows from Theorem \ref{wr} and Proposition \ref{contra}.
\end{proof}

%\begin{corollary}
 %$$A(C_n \wr H)=A(H)^{\times n}\times A(C_n),$$
 %for $n>2$.
%\end{corollary}

Finally, we present a characterization of connectedness for the antipodal graph of a wreath product.

\begin{prop}\label{proplabnew}
The graph $A(G \wr H)$ is connected if and only if $A(H)$ and $A_{Ha}(G)$ are connected and $A(H)$ is not bipartite.
\end{prop}
\begin{proof}
It is known (see Corollary 5.10 in \cite{imrichbook}) that the direct product of connected graphs is connected if and only if at most one of the factors is bipartite.
\end{proof}

\begin{example}\label{antipk2k3}
In Fig. \ref{figureantipodal1} we have represented the graph $K_2\wr C_3$, where we have put $V_{K_2}=\{x_1,x_2\}$ and $V_{C_3}=\{0,1,2\}$. It is possible to directly check that $diam(K_2\wr C_3)=4$, as expected by Corollary \ref{diam},
since $diam_{Ha}(K_2)=2 $ and $diam(C_3)=1$.
\begin{figure}[h]
\begin{center}
\psfrag{00a}{\tiny{$(00)x_1$}}\psfrag{00b}{\tiny{$(00)x_2$}}\psfrag{01a}{\tiny{$(01)x_1$}}\psfrag{01b}{\tiny{$(01)x_2$}}\psfrag{02a}{\tiny{$(02)x_1$}}\psfrag{02b}{\tiny{$(02)x_2$}}
\psfrag{10a}{\tiny{$(10)x_1$}}\psfrag{10b}{\tiny{$(10)x_2$}}\psfrag{11a}{\tiny{$(11)x_1$}}\psfrag{11b}{\tiny{$(11)x_2$}}\psfrag{12a}{\tiny{$(12)x_1$}}\psfrag{12b}{\tiny{$(12)x_2$}}
\psfrag{20a}{\tiny{$(20)x_1$}}\psfrag{20b}{\tiny{$(20)x_2$}}\psfrag{21a}{\tiny{$(21)x_1$}}\psfrag{21b}{\tiny{$(21)x_2$}}\psfrag{22a}{\tiny{$(22)x_1$}}\psfrag{22b}{\tiny{$(22)x_2$}}
\includegraphics[width=0.5\textwidth]{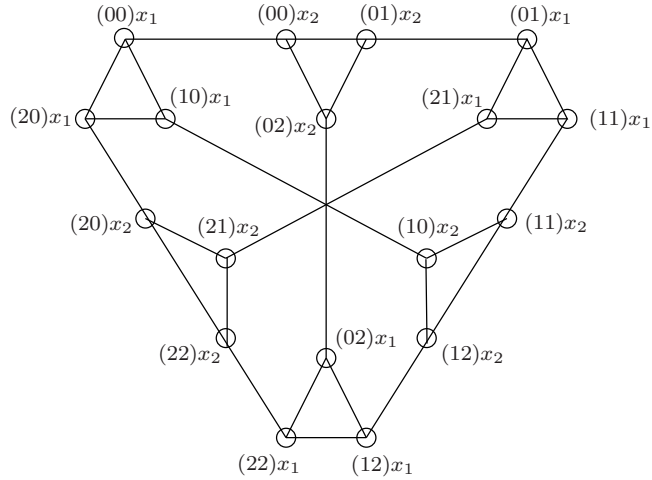}
\end{center}\caption{The graph $K_2\wr C_3$.} \label{figureantipodal1}
\end{figure}

Two vertices are at distance $4$ if and only if they share the position ($x_1$ or $x_2$, in $K_2$) but their configurations differ in both the coordinates ($0$, $1$ or $2$, in $C_3$).
By this explicit analysis, or equivalently by recalling Corollary \ref{Hami} and remembering that $A(C_3)=C_3$, we deduce $A(K_2\wr C_3) = C_3 \times C_3 \times O_2$. The graph $A(K_2\wr C_3)$ is depicted in Fig. \ref{figureantipodal2}. Observe that the direct product with $O_2$ is nothing but a disjoint duplication of the graph $C_3 \times C_3$, which coincides with the collinearity graph of the \emph{Generalized Quadrangle} $GQ(2,1)$.

\begin{figure}[h]
\begin{center}
\psfrag{00a}{\tiny{$(01)x_1$}}\psfrag{00b}{\tiny{$(01)x_2$}}\psfrag{01a}{\tiny{$(10)x_1$}}\psfrag{01b}{\tiny{$(10)x_2$}}\psfrag{02a}{\tiny{$(02)x_1$}}\psfrag{02b}{\tiny{$(02)x_2$}}
\psfrag{10a}{\tiny{$(11)x_1$}}\psfrag{10b}{\tiny{$(11)x_2$}}\psfrag{11a}{\tiny{$(20)x_1$}}\psfrag{11b}{\tiny{$(20)x_2$}}\psfrag{12a}{\tiny{$(12)x_1$}}\psfrag{12b}{\tiny{$(12)x_2$}}
\psfrag{20a}{\tiny{$(21)x_1$}}\psfrag{20b}{\tiny{$(21)x_2$}}\psfrag{21a}{\tiny{$(00)x_1$}}\psfrag{21b}{\tiny{$(00)x_2$}}\psfrag{22a}{\tiny{$(22)x_1$}}\psfrag{22b}{\tiny{$(22)x_2$}}
\includegraphics[width=0.7\textwidth]{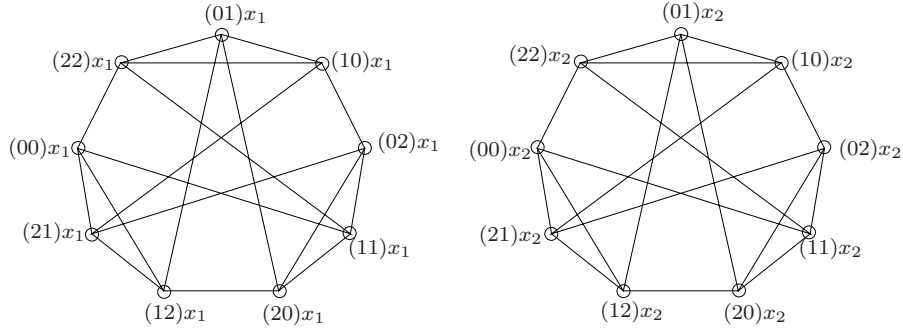}
\end{center}\caption{The graph $A(K_2\wr C_3)$.}  \label{figureantipodal2}
\end{figure}
\end{example}

\begin{example}
Recall from Example \ref{exclas} that $A(K_n)=K_n$, that for and odd $n$ we have $A(C_n)=C_n$, and that $A(P_m)=K_2 \sqcup  \{\mbox{$m-2$ isolated vertices}\}$.
For the properties of the direct product we have, for every graphs $G,G_1,H_1,G_2,H_2:$
$$
(G_1\sqcup H_1) \times (G_2\sqcup H_2)=(G_1\times G_2)  \sqcup (G_1\times H_2) \sqcup (H_1\times G_2) \sqcup (H_1\times H_2),
$$
$$
G\times O_m= \bigsqcup_{i=1}^{m} G=m G,
$$
$$
G\times \{\mbox{isolated vertex}\}=|V_G|\mbox{ isolated vertices}.
$$
In particular
\begin{equation*}
\begin{split}
A(P_m)^{\times 2}=(K_2 \sqcup  \{\mbox{$m-2$ isolated vertices}\})^{\times 2} &= (K_2\times K_2) \sqcup \{\mbox{$m^2-4$ isolated vertices}\}\\
&=K_2\sqcup K_2 \sqcup \{\mbox{$m^2-4$ isolated vertices}\}
\end{split}
\end{equation*}
and more generally
$$
A(P_m)^{\times n}=(K_2 \sqcup  \{\mbox{$m-2$ isolated vertices}\})^{\times n}= 2^{n-1} K_2 \sqcup  \{\mbox{$m^n-2^n$ isolated vertices}\}.
$$
Moreover, it is easy to prove that, for odd $n$, $K_2\times C_n=C_{2n}$. By combining Theorem \ref{wr}, Corollary \ref{Hami}, and Example \ref{esempi}, we are able to describe the antipodal graphs of several families of graphs:
\begin{itemize}
\item
$A(C_n\wr K_m)=K_m^{\times n} \times C_n$, if $n\geq7$ odd;
\item $A(C_n\wr C_m)=C_m^{\times n} \times C_n$, if $m,n$ are odd and $n\geq 7$;
%\item  $A(C_n\wr P_m)=2^{n-1} K_2\times C_n\sqcup \{ \mbox{ $n(m^n-2^n)$ isolated vertices}\}$, if $n\geq7$ odd;
\item  $A(C_n\wr P_m)=2^{n-1}C_{2n}\sqcup \{\mbox{$n(m^n-2^n)$ isolated vertices}\}$, if $n\geq7$ odd;
\item
$A(K_n\wr C_m)=A(P_n\wr C_m)=C_m^{\times n} \times O_n=n C_{m}^{\times n}$, if $m$ is odd;
\item
$A(K_n\wr K_m)=A(P_n\wr K_m)=K_m^{\times n}\times O_n= n K_{m}^{\times n}$;
\item
$A(K_n\wr P_m)=A(P_n\wr P_m)= n 2^{n-1} K_2 \sqcup \{\mbox{$n(m^n-2^n)$ isolated vertices}\}$.
\end{itemize}
\end{example}

\section{Wiener index of a wreath product}\label{sectionwiener}
The Wiener index is a distance-based topological index, introduced by H. Wiener in \cite{wiener} for graphs associated with molecules.
With a breadth-first search, it is possible to compute the Wiener index of an arbitrary graph with $n$ vertices and $m$ edges in time $O(nm)$. However, in graphs that are products of smaller graphs, especially for the wreath product, the number of vertices and edges could be very large: a better approach, that we will use in this section, consists in relating the Wiener index of a product to the Wiener index of the factor graphs.

\begin{defi}
Let $G=(V_G,E_G)$ be a connected graph. The \textit{Wiener index} of $G$ is  defined as the sum of the distances between all the unordered pairs of vertices, i.e.,
$$
W(G)= \frac{1}{2}\sum_{u,v\in V_G} d_{G}(u,v).
$$
\end{defi}
The Wiener index is an isomorphism invariant, that is, if $G_1$ and $G_2$ are isomorphic graphs, then $W(G_1)=W(G_2)$. Of course, it is not a \emph{complete invariant}, that is, the converse implication is false. The easiest counterexample is the pair of non-isomorphic graphs given by the cycle $C_4$ and the \emph{paw graph} $P$. The paw graph is a graph with vertex set $V_P=\{a, b, c, d\}$ and edge $E_P=\{\{a,b\},\{b,c\},\{a,c\},\{c,d\}\}$ (see Fig. \ref{cyclicandpaw}). It is easy to see that $W(C_4)=W(P)=8$.

\begin{figure}[h]
\begin{center}
\psfrag{1}{$1$}\psfrag{2}{$2$}\psfrag{3}{$3$}\psfrag{4}{$4$}
\psfrag{a}{$a$}\psfrag{b}{$b$}\psfrag{c}{$c$} \psfrag{d}{$d$}
\includegraphics[width=0.5\textwidth]{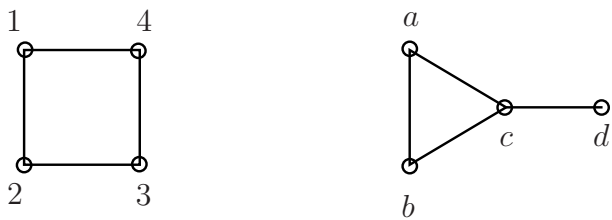}
\end{center}\caption{The graph $C_4$ and the Paw graph $P$.}\label{cyclicandpaw}
\end{figure}

However, it is true that, if the graph $G$ is such that $diam(G)=2$ and $|V_G|=n$, then we have $|E_G|$ pairs of vertices at distance $1$ and $\frac{n(n-1)}{2}-|E_G|$ pairs of vertices at distance $2$. Therefore $W(G)=n(n-1)-|E_G|$. This implies that two graphs of diameter $2$, with the same number of vertices and the same number of edges, have the same Wiener index.

In \cite{wprod} an explicit computation of the Wiener index for a Cartesian product is provided. For the sake of completeness, and in view of Theorem \ref{ww}, we study in Lemma \ref{carte} the particular case of the Wiener index of a Cartesian power.

\begin{lemma}\label{carte}
Let $H=(V_H,E_H)$ be a connected graph, with $|V_H|=m$. Then, for every $n\geq 1$, we have:
$$
W(H^{\Box n})=n m^{2(n-1)} W(H).
$$
\end{lemma}
\begin{proof}
Set $x=(x_1,\ldots, x_n)$, $x'=(x'_1,\ldots, x'_n)$, with $x_i, x'_i\in V_H$ for $i=1,\ldots, n$. Then
$$
W(H^{\Box n})=\frac{1}{2} \sum_{x,x'\in V_{H^{\Box n}}} d_{H^{\Box n}}(x,x')=\frac{1}{2} \sum_{x,x'\in V_{H^{\Box n}}} \sum_{i=1}^n d_{H}(x_i,x_i').
$$
For any pair $u,v\in V_H$, the contribution given by $d_H(u,v)$ appears exactly $m^{2(n-1)}$ times for each position $i$, and it gives:
$$
W(H^{\Box n})=\frac{1}{2} n m^{2(n-1)} \sum_{u,v\in V_H} d_H(u,v) = n m^{2(n-1)} W(H).
$$
\end{proof}

\begin{defi}\label{wa}
Let $G=(V_G,E_G)$ be a connected graph and let $A\subseteq V_G$, we put $W_{\rho_A}(G):=\frac{1}{2} \sum_{u,v\in V_G} \rho_A(u,v).$
\end{defi}
%Observe that, by Proposition \ref{inva}, we have that $W_{\rho_A}(G)=W_{\rho_{\phi(A)}}(G)$, for every $\phi\in Aut(G)$.

%$W_{\emptyset}(G)=W(G)$. For any $u^*\in V_G$ we have $W_{\{u^*\}}(G)=\frac{1}{2} \sum_{u,v\in V_G} \rho_{\{u^*\}}(u,v)=\frac{1}{2} \sum_{u,v\in V_G} {d_G(u,u^*}+d_G(u^*,v)$

\begin{theorem}\label{ww}
Let $G=(V_G,E_G)$ and $H=(V_H,E_H)$ be two connected graphs with $|V_G|=n$, $|V_H|=m$. Then:
\begin{equation}\label{wiwr}
W(G\wr H)= n^3 m^{2(n-1)} W(H) + m^n \sum_{A\subseteq V_G} (m-1)^{|A|} W_{\rho_A}(G).
\end{equation}
\end{theorem}
\begin{proof}
Let $u=(y,x)\in V_{G\wr H} $ and $u'=(y',x')$ in $V_{G\wr H}$, with $y,y'\in V_{H^{\Box n}}$ and $x,x'\in V_{G}$. Then, by Theorem \ref{dista}, we have
$$
W(G\wr H)= \frac{1}{2}\sum_{u,u'\in V_{G\wr H}}{(d_{H^{\Box n}}(y,y')+\rho_{\delta(y,y')}(x,x'))}.
$$
Fixing a pair $y, y' \in H^{\Box n}$, the contribution $d_{H^{\Box n}}(y,y')$ appears $n^2$ times in the sum so that, using Lemma \ref{carte}, we get:
$$
\frac{1}{2}\sum_{u,u'\in V_{G\wr H}} d_{H^{\Box n}}(y,y')= \frac{1}{2}\sum_{y,y'\in V_{H^{\Box n}}} n^2 d_{H^{\Box n}}(y,y') = n^2 W(H^{\Box n})= n^3 m^{2(n-1)}W(H).
$$
On the other hand, for any fixed subset $A\subseteq V_G$, we have:
$$
|\{(y,y')\in (V_{H^{\Box n}})^2: \delta(y,y')=A\}|= m^n (m-1)^{|A|},
$$
as we can freely choose the $n$ coordinates of $y$ among the $m$ elements of $V_H$, and we can choose $|A|$ coordinates of $y'$ among the $(m-1)$ elements of $V_H$ (they have to be different from the ones chosen  for $y$). This implies that, for every $x,x'\in V_G$ and any fixed $A\subseteq V_G$, the contribution given by $\rho_A(x,x')$ appears $m^n (m-1)^{|A|}$ times in the sum, so that:
$$
\frac{1}{2}\sum_{u,u'\in V_{G\wr H}} \rho_{\delta(y,y')}(x,x')=  \frac{1}{2} \sum_{A\subseteq V_G} \sum_{x,x'\in V_G} m^n (m-1)^{|A|} \rho_A(x,x').
$$
Then the claim is proven according to Definition \ref{wa}.
\end{proof}

Notice that the coefficient of the term $W_{\rho_A}(G)$ in  Equation \eqref{wiwr} of Theorem \ref{ww} only depends on the cardinality of $|A|$. This motivates the following definition.
\begin{defi}\label{wv}
Let $G=(V_G,E_G)$ be a connected graph with $|V_G|=n$.
For any $k\in \{0,1,\ldots,n\}$, set
$$
W_{\rho_k}(G):=\sum_{A\subseteq V_G, \, |A|=k} W_{\rho_A}(G).
$$
Moreover, we call \emph{Wiener vector} of $G$ the $(n+1)$-vector:
$$
W_{\rho}(G):=(W_{\rho_0}(G), W_{\rho_1}(G),\ldots, W_{\rho_n}(G)).
$$
\end{defi}
\noindent The following identities hold:
\begin{equation*}\begin{split}
W_{\rho_0}(G)&= W_{\rho_\emptyset}(G)=W(G);\qquad \qquad  W_{\rho_n}(G)= W_{\rho_{V_G}}(G);\\
W_{\rho_1}(G)&=\sum_{u^* \in V_G}
W_{\rho_{\{u^*\}}}(G)=\sum_{u^*\in V_G}
\frac{1}{2} \sum_{u,v\in V_G} \rho_{\{u^*\}}(u,v)\\&=
\frac{1}{2}\sum_{u^*\in V_G}
\sum_{u,v\in V_G}(d_G(u, u^*)+
d_G (u^*, v))=2 n W(G). \end{split}
\end{equation*}

\begin{corollary}\label{wwk}
Let $G=(V_G,E_G)$ and $H=(V_H,E_H)$ be connected graphs with $|V_G|=n$, $|V_H|=m$. Then we have
\begin{eqnarray}\label{formjan}
W(G\wr H)= n^3 m^{2(n-1)} W(H) + m^n \sum_{k=0}^{n} (m-1)^{k} W_{\rho_k}(G).
\end{eqnarray}
\end{corollary}
\begin{proof}
It follows from Theorem \ref{ww} and from the definition of $W_{\rho_k}(G)$.
\end{proof}
It follows that the vector $W_{\rho}(G)$ is an invariant for the graph that completely determines the Wiener index of the wreath products where $G$ is the first factor. In other words, for any pair of connected graphs $G_1$ and $G_2$ such that $W_{\rho}(G_1)=W_{\rho}(G_2)$, we have $W(G_1\wr H)=W(G_2\wr H)$ for any connected graph $H$.
Actually, also the converse is true, as the following proposition shows.
\begin{prop}
For any given pair of connected graphs $G_1$ and $G_2$, we have:
$$
W_{\rho}(G_1)=W_{\rho}(G_2) \iff  W(G_1\wr H)=W(G_2\wr H) \mbox{ for any connected graph $H$}.
$$
\end{prop}
\begin{proof}
It is enough to regard the term $m^n \sum_{k=0}^{n} (m-1)^{k} W_{\rho_k}(G)$ in \eqref{formjan} as a polynomial in the variable $m$ and to use the classical fact that if the evaluations of two polynomials coincide in a number of points larger than the degree, then the two polynomials are the same. Since $W(G_1\wr H)=W(G_2\wr H)$ for any connected graph $H$, in particular we have $W(G_1\wr K_i)=W(G_2\wr K_i)$ for $i=1,2,\ldots, n+1$, and this implies $W_{\rho}(G_1)=W_{\rho}(G_2)$.
\end{proof}

It is an easy but tedious exercise to prove that the graph $C_4$ and the paw graph $P$ have different Wiener vectors:
$$
W_{\rho}(C_4)=(8, 64,132,104, 28) \qquad W_{\rho}(P)=(8, 64,134, 110, 32).
$$
It is natural to ask about the nature of this invariant.

\begin{question}
Are there  pairs of non-isomorphic connected graphs $G_1$, $G_2$, with $W_{\rho}(G_1)=W_{\rho}(G_2)$? Equivalently, are there pairs of non-isomorphic connected graphs $G_1$, $G_2$, such that $W(G_1\wr H)=W(G_2\wr H)$  for any connected graph $H$?
\end{question}

In the remaining part of the section, we analyze the special cases of the wreath products $K_n\wr H$ and $P_n\wr H$.

\subsection{Wiener index of $K_n \wr H$}\label{sub1}
Let us consider the case $G=K_n$. The following lemma holds.

\begin{lemma}\label{vectc}
The components of the Wiener vector of the complete graph $K_n$ are:
$$
W_{\rho_0}(K_n)= \frac{n(n-1)}{2} \qquad \qquad W_{\rho_1}(K_n)= n^2(n-1)
$$
$$
W_{\rho_k}(K_n)=  \frac{1}{2}{n\choose k} (kn^2-2kn+k+n^2),\quad \textrm{ with } 2\leq k\leq n.
$$
\end{lemma}
\begin{proof}
Notice that, since in the complete graph $K_n$ all subsets of the same cardinality are isomorphic, the quantity $W_{\rho_A}(K_n)$ only depends on the cardinality of the subset $A$. Let $A\subseteq V_{K_n}$, with $|A|=k\geq 2$. Then it is easy to see that
\begin{equation*}\begin{split}
&|\{(u,v)\in (V_{K_n})^2: (u\in A, v\notin A)\mbox{ or } (v\in A, u\notin A)\}|=2k(n-k),\\
&|\{(u,v)\in (V_{K_n})^2: u\notin A, v\notin A\}|=(n-k)^2,\\
&|\{(u,v)\in (V_{K_n})^2: u\in A, v\in A, u\neq v\}|=k(k-1),\\
&|\{(u,v)\in (V_{K_n})^2: u\in A, v\in A, u=v\}|=k.\\
\end{split}
\end{equation*}
Therefore, by using Equation \eqref{rocompleto} from Example \ref{K_n} and Definition \ref{wa}, we obtain:
$$
W_{\rho_A}(K_n)=\frac{1}{2}\left(kn^2-2kn+k+n^2\right).
$$
The claim follows.
\end{proof}

We are now in position to prove the following theorem.
\begin{theorem}\label{wcom}
Let $H$ be a connected graph with $|V_{H}|=m$. Then:
$$
W(K_n\wr H)=n^3 m^{2n-2}W(H)+\frac{nm^{n}}{2}(n^2m^n-n^2m^{n-1}-m^{n}n+2m^{n-1}n-m+m^{n}-m^{n-1}).
$$
\end{theorem}
\begin{proof}
By virtue of Corollary \ref{wwk} and Lemma \ref{vectc}, we have:
\begin{equation}\label{conti}\begin{split}
W(K_n\wr H)=& n^3 m^{2n-2}W(H)+\frac{n(n-1)}{2} m^n +n^2(n-1) m^n (m-1) \\+&m^n \sum_{k=2}^n {(m-1)^k \frac{1}{2}{n\choose k} (kn^2-2kn+k+n^2)}.
\end{split}
\end{equation}
By the Binomial Theorem $\sum_{k=0}^n {(m-1)^k {n\choose k}}=m^n$, we deduce:
$$
\sum_{k=0}^n {(m-1)^k {n\choose k}k}= n(m-1)  \sum_{k=1}^n {(m-1)^{k-1} {{n-1}\choose{k-1}}}= n (m-1)  m^{n-1},
$$
so that we can rewrite the last term in Equation \eqref{conti} as follows:
\begin{equation*}
\begin{split}
&m^n \sum_{k=2}^n {(m-1)^k \frac{1}{2}{n\choose k} (kn^2-2kn+k+n^2)}=\\
 %m^n \sum_{k=2}^n {(m-1)^k \frac{1}{2}{n\choose k} k(n^2-2n+1)+n^2}&=\\
&\frac{1}{2}m^n(n^2-2n+1) \sum_{k=2}^n {(m-1)^k {n\choose k}k}
+ \frac{1}{2}m^n n^2 \sum_{k=2}^n {(m-1)^k {n\choose k}}=\\
&\frac{1}{2}m^n(n^2-2n+1)[n(m-1)m^{n-1}-n(m-1)]+ \frac{1}{2}m^n n^2 [m^n-n(m-1)-1].
\end{split}
\end{equation*}
Finally, by summing all the contributions, we get the claim.
\end{proof}
Observe that Theorem \ref{wcom} generalizes the result of \cite{compcomp} (case $H=K_m$) and of \cite{BCD} (case $H=C_m$).

\subsection{Wiener index of $P_n \wr H$}\label{sub2}
We know that $W_{\rho_A}(P_n)=W_{\rho_{\{\min A,\max A\}
}}(P_n)$ from Example \ref{P_n}. With every fixed set $A$, we associate the numbers
$$
a:=\min A-1, \qquad b:=\max A-\min A+1, \qquad c=n-\max A,
$$
and the following partition $\mathcal P_{a|b|c}$ of the vertex set (note that $a+b+c=n$):
$$
V_{P_n}=\{1,2,\ldots, n\}=\{1,2,\ldots, \min A-1\} \sqcup \{\min A, \ldots, \max A\} \sqcup \{\max A+1, \ldots, n\},
$$
whose parts have cardinality $a$, $b$, $c$, respectively. This induces a partition of $(V_{P_n})^2$ into $9$ parts. In the next proposition, we will give an explicit expression of $W_{\rho_{A}}(P_n)$, as the sum of the contributions of each of these parts.

\begin{prop}
Let $P_n$ be the path graph, with $n>2$. Let $A\subseteq V_{P_n}$ and $a:=\min A-1$, $b:=\max A-\min A+1$. Then we have:
\begin{equation}\label{mapath}
W_{\rho_A}(P_n)=\frac{1}{2}\left(n^3-n^2+\frac{1}{3}b(2b-1)(b-1)-2a(n-b-a)(n+b-1)\right).
\end{equation}
\end{prop}
\begin{proof}
For every $x,n\in \mathbb N$, let us denote with $\mathcal O(x,n)$ the $n\times n$ matrix such that $\mathcal O(x,n)$ is symmetric, $\mathcal O(x,n)_{1,1}=x$, and
\begin{equation*}
\mathcal O(x,n)_{i+1,j}=\begin{cases}
\mathcal O(x,n)_{i,j}+1 \mbox{ if } i+j\leq n\\
\mathcal O(x,n)_{i,j}-1 \mbox{ otherwise. }
\end{cases}
\end{equation*}
More explicitly, we have:
$$
\mathcal O(x,n)=\left(
                  \begin{array}{ccccc}
                    x & x+1 & \cdots & x+n-2 & x+n-1 \\
                     x+1 & x+2 & \cdots & \cdots & x+n-2 \\
                    \vdots & \vdots & \ddots &  & \vdots \\
                    x+n-2 & x+n-1 &  & \ddots & x+1 \\
                    x+n-1 & x+n-2 & \cdots & x+1 & x \\
                  \end{array}
                \right)
$$
Let us denote with $|\mathcal O(x,n)|$ the sum of all entries of $\mathcal O(x,n)$. It is easy to check that
\begin{equation}\label{O}|\mathcal O(x,n)|=(x-1)n^2+\frac{1}{3}(2n^3+n).
\end{equation}
In particular $|\mathcal O(1,n)|=\frac{n(2n^2+1)}{3}$ is the sequence of the so called \emph{ Octahedral numbers} (sequence A005900 in the OEIS).\\
\indent For every $x,n,m\in \mathbb N$, let us denote with $\mathcal V(x,n,m)$ the $n\times m$ matrix such that $\mathcal V(x,n,m)_{1,1}=x$ and $\mathcal V(x,n,m)_{i+1,j}=\mathcal V(x,n,m)_{i,j}+1$, for every $i=1,\ldots, n-1$ and $j=1,\ldots,m$, and $\mathcal V(x,n,m)_{i,j+1}=V(x,n,m)_{i,j}+1$, for every $i=1,\ldots, n$ and $j=1,\ldots,m-1$. More explicitly, we have:
$$
\mathcal V(x,n,m)=\left(
                \begin{array}{cccc}
                  x & x+1 & \cdots & x+m-1 \\
                  x+1 & x+2 & \cdots & x+m\\
                  \vdots & \vdots & \ddots & \vdots \\
                  x+n-1 & x+n& \cdots & x+n+m-2 \\
                \end{array}
              \right).
$$
Let us denote with $|\mathcal V(x,n,m)|$ the sum of all entries of $\mathcal V(x,n,m)$. One can check that
\begin{equation}\label{V}
|\mathcal V(x,n,m)|=xnm+\frac{nm}{2}(n+m-2).
\end{equation}
Finally, let us denote by $p_n$ the $n\times n$ permutation matrix with $1$ on the antidiagonal entries, that is,
$$
(p_n)_{i,j} = \left\{
                \begin{array}{ll}
                  1 & \hbox{if } i+j=n+1 \\
                  0 & \hbox{otherwise.}
                \end{array}
              \right.
$$
Given an $n\times m$ matrix $B$, the matrices $p_n B$, $B p_m$, $p_n B p_m$ are nothing but all possible rotations of the matrix $B$. In particular, we have \begin{equation}\label{rot}
|B|=|p_n B|=|B p_m|=|p_n B p_m|.
\end{equation}
Now let $\rho_A$ be the $n\times n$ symmetric matrix such that $({\rho_A})_{i,j}=\rho_A(i,j)$. Then $2W_{\rho_A}(P_n)=|\rho_A|$. The matrix $\rho_A$ consists of $9$ blocks determined by the partition $\mathcal P_{a|b|c}$ associated with $A$: looking at Equation \eqref{patpat} in Example \ref{P_n}, it can be seen that these blocks are exactly, up to rotations, the matrices introduced above:
$$
\rho_A=\left(
  \begin{array}{c|c|c}
 p_a\mathcal V(2b,a,a) p_a & p_a \mathcal V(b,a,b) p_b & p_a \mathcal V(b+1,a,c)\\
    \hline
      p_b \mathcal V(b,b,a) p_a &p_b \mathcal  O(b-1,b)& \mathcal V(b,b,c)\\
    \hline
       \mathcal V(b+1,c,a)p_a &  \mathcal V(b,c,b) &  \mathcal V(2b,c,c)\\
  \end{array}
\right).
$$
Combining this description of the matrix $\rho_A$ with Equation \eqref{rot}, we obtain:
\begin{equation*}
2 W_{\rho_A}(P_n)=|\mathcal V(2b,a,a)|+2|\mathcal V(b,a,b)|+2|\mathcal V(b+1,a,c)|+2|\mathcal V(b,b,c)|+|\mathcal V(2b,c,c)|+ |\mathcal O(b-1,b)|
\end{equation*}
By a direct computation, which makes use of Equations \eqref{O} and \eqref{V}, we get the claim.
\end{proof}
We are now in position to compute the Wiener vector of the graph $P_n$.

\begin{theorem}\label{tp}
Let $P_n$ be the path graph, with $n>2$. Then, for any $k\in\{0,1,\ldots, n\}$, we have:
$$W_{\rho_k}(P_n)={{n+1}\choose{k+1}}\frac{5k^3n^2+k^3n+18k^2n^2-18k^2n-12k^2+19kn^2-25kn+12k+6n^2-6n}{6(k+2)(k+3)}.$$
\end{theorem}
\begin{proof}
Once fixed three non-negative integers $a$, $b$, $c$ such that $a+b+c=n$, there are exactly ${{b-2}\choose{k-2}}$ subsets $A\subseteq \{1,2,\ldots,n\}$ with $|A|=k$ and such that  $a=\min A-1$, $b=\max A-\min A+1$, $c=n-\max A$, provided that $b\geq k$. Combining with Equation \eqref{mapath}, we have
$$
W_{\rho_k}(P_n)=\sum_{a=0}^{n-k}\sum_{b=k}^{n-a} {{b-2}\choose{k-2}}\frac{1}{2}\left(n^3-n^2+\frac{1}{3}b(2b-1)(b-1)-2a(n-b-a)(n+b-1)\right).
$$
Finally, by an explicit computation, we get the claim.
\end{proof}
Notice that Theorem \ref{tp} presents an explicit computation of the vector $W_{\rho}(P_n).$ By applying the formula given in Corollary \ref{wwk}, it is possible to compute $W(P_n\wr H)$, for every connected graph $H$ whose Wiener index $W(H)$ is known.

\section{Szeged index of a wreath product}\label{sectionszeged}

We start this section by recalling the definition of Szeged index for a connected graph $G=(V_G, E_G)$, as presented in \cite{szeged1}. Given an edge $e=\{u,v\}\in E_G$ (in other words, $u$ and $v$ are adjacent vertices in $G$), one defines:
$$
B_u (e) = \{w \in V_G : d_G(w,u) < d_G(w,v)\} \qquad  B_v (e) = \{w \in V_G : d_G(w,v) < d_G(w,u)\}.
$$
In particular, if $d_G(w,u)=d_G(w,v)$, then $w$ is neither in $B_u(e)$ nor in $B_v(e)$. Finally, one defines $n_u(e) = |B_u(e)|$ and $n_v(e) = |B_v(e)|$. The Szeged index $Sz(G)$ of $G$ is defined as:
\begin{eqnarray*}
Sz(G) = \sum_{e=\{u,v\}\in E_G} n_u(e)n_v(e).
\end{eqnarray*}
It is known that, if $G$ is a tree, then the Szeged index $Sz(G)$ coincides with the Wiener index $W(G)$. In \cite{szeged1}, a more general sufficient condition for a graph $G$ to satisfy the equality $Sz(G) = W(G)$ is given. In \cite{szeged2}, some connections between the Wiener index and the Szeged index are investigated, and the Szeged index of some graph compositions are described. We also want to mention that the Szeged index and a weighted generalization of it are studied for graph compositions in \cite{szeged3}. The aim of this section is to study the Szeged index for a wreath product of two graphs.

\begin{example}\label{szK_n}
Let $K_n$ be the complete graph on $n$ vertices. For any $e=\{u,v\}\in E_G$, the set $B_u (e)$ contains only the vertex $u$, so that:
$$
Sz(K_n) = \sum_{e=\{u,v\}\in E_{K_n}} n_u(e)n_v(e)=\frac{1}{2} n (n-1).
$$
\end{example}

\begin{example}\label{szC_m}
Let $C_m$ be the cycle graph on $m$ vertices. For any $e=\{u,v\}\in E_G$, the set $B_u (e)$ contains  $\lfloor \frac{m}{2}\rfloor$  vertices. Then we have:
$$
Sz(C_m) = \sum_{e=\{u,v\}\in E_{C_m}} n_u(e)n_v(e)=m \left\lfloor \frac{m}{2}\right \rfloor^2.
$$
\end{example}

Let $G=(V_G,E_G)$ and $H=(V_H,E_H)$ be two connected graphs. In order to compute $Sz(G\wr H)$ we decompose $E_{G\wr H}$ into the subset of edges of type I, denoted by $E_I$, and the subset of edges of type II, denoted by $E_{II}$. Thus, we put
$$
Sz_I(G\wr H):=\sum_{e=\{u,v\}\in E_I} n_u(e)n_v(e) \qquad \qquad Sz_{II}(G\wr H):=\sum_{e=\{u,v\}\in E_{II}} n_u(e)n_v(e),
$$
so that
\begin{equation}\label{split}
Sz(G\wr H)=Sz_I(G\wr H)+Sz_{II}(G\wr H).
\end{equation}
Let us start by considering edges of type I.
\begin{lemma}\label{n1}
Consider two vertices $u,v\in V_{G\wr H}$, with $u=(w_1,\ldots, w_{k-1}, a, w_{k+1}, \ldots, w_n)x_k$ and  ${v=(w_1,\ldots, w_{k-1}, b, w_{k+1}, \ldots, w_n)x_k}$, with $a\sim b$ in $H$. Let $z=(y_1,\ldots, y_n)x_i$ be an arbitrary vertex of $V_{G\wr H}$. Then we have
$$
d_{G\wr H}(z, u)> d_{G\wr H}(z,v)\iff d_H(y_k,a)>d_H(y_k,b).
$$
\end{lemma}
\begin{proof}
By virtue of Theorem \ref{dista}, we have:
$$
d_{G\wr H}(z, u)=\sum_{j=1, \, j\neq k}^n d_H(y_j,w_j)+d_H(y_k,a)+ \rho_{A}(x_i, x_k),
$$
$$
d_{G\wr H}(z, v)=\sum_{j=1, \, j\neq k}^n d_H(y_j,w_j)+d_H(y_k,b)+ \rho_{B}(x_i, x_k)
$$
where
$$
A=\delta((y_1,\ldots, y_n), (w_1,\ldots, w_{k-1}, a, w_{k+1}, \ldots, w_n))
$$
and similarly
$$
B=\delta((y_1,\ldots, y_n), (w_1,\ldots, w_{k-1}, b, w_{k+1}, \ldots, w_n)).
$$
If $y_k\neq a$ and $y_k\neq b$, then $A=B$ and the claim is true.\\
If $y_k=a$ and $y_k\neq b$, then $0=d_H(y_k,a)<d_H(y_k,b)$ and $A\subset B$, thus $\rho_A\leq \rho_B$ by monotonicity and then $d_{G\wr H}(z, u)< d_{G\wr H}(z, v)$: the claim is true. If $y_k\neq a$ and $y_k= b$, $d_H(y_k,a)>d_H(y_k,b)=0$ and $B\subset A$, thus $\rho_B\leq \rho_A$ by monotonicity and then $d_{G\wr H}(z, u)> d_{G\wr H}(z, v)$: the claim is true.
\end{proof}

\begin{lemma}\label{n1e}
Consider two vertices $u,v\in V_{G\wr H}$, with $u=(w_1,\ldots, w_{k-1}, a, w_{k+1}, \ldots, w_n)x_k$ and  ${v=(w_1,\ldots, w_{k-1}, b, w_{k+1}, \ldots, w_n)x_k}$, with $a\sim b$ in $H$. Let us put $E=\{u,v\}\in E_{G\wr H}$ and $e=\{a,b\}\in E_H$. Then we have:
$$
n_u(E)=n m^{n-1}  n_a(e).
$$
\end{lemma}
\begin{proof}
By Lemma \ref{n1} a vertex $z=(y_1,\ldots, y_n)x_i\in V_{G\wr H}$ is in $B_u(E)$ if and only if $y_k$ is in $B_a(e)$. For a fixed $y\in V_H$, there are exactly $n m^{n-1}$ vertices $z=(y_1,\ldots, y_n)x_i\in V_{G\wr H}$ such that $y_k=y$, that is, $|B_u(E)|= n m^{n-1}  |B_a(e)|$. The claim follows.
\end{proof}

\begin{prop}\label{n1p}
Let $G=(V_G,E_G)$ and $H=(V_H,E_H)$ be connected graphs with $|V_G|=n$, and $|V_H|=m$. Then
$$ Sz_{I}(G\wr H)=n^3 m^{3n-3} Sz(H).$$
\end{prop}
\begin{proof}
By virtue of Lemma \ref{n1e}, setting $z=(y_1,\ldots,y_n)x_i$, we have:
$$
Sz_{I}(G\wr H)=\sum_{E=\{u,v\}\in E_I} n_u(E)n_v(E)=\sum_{z \in V_{G\wr H},\, b\in V_H: e=\{y_i,b\}\in E_H} n_{y_i}(e) n_b(e) n^2 m^{2n-2}.
$$
For a fixed edge $e=\{a,b\} \in E_H$, there are exactly $n m^{n-1}$ vertices $z=(y_1,\ldots, y_n)x_i\in V_{G\wr H}$ such that $y_i=a$, then
$$
Sz_{I}(G\wr H)= n m^{n-1}  \sum_{e=\{a,b\}\in E_H}n_a(e) n_b(e) n^2 m^{2n-2}=n^3 m^{3n-3} Sz(H).
$$
\end{proof}
We pass now to the investigation of edges of type II. Also in this case, the maps $\rho_A$, with $A\subseteq V_G$, will play a crucial role.
\begin{defi}
Let $G=(V_G,E_G)$ be a connected graph, and let $A\subseteq V_G$ and $e=\{x_j,x_k\}\in E_G$. We put
$$
B_{x_j}(e,\rho_A):=\{x_i\in V_G: \rho_A(x_i,x_j)<\rho_A(x_i,x_k)\} \qquad \textrm{and } \ n_{x_j}(e,\rho_A)=|B_{x_j}(e,\rho_A)|;
$$
$$
B_{x_k}(e,\rho_A):=\{x_i\in V_G: \rho_A(x_i,x_k)<\rho_A(x_i,x_j)\} \qquad \textrm{and } \ n_{x_k}(e,\rho_A)=|B_{x_k}(e,\rho_A)|.
$$
\end{defi}

\begin{lemma}\label{n2}
Consider two vertices $u,v\in V_{G\wr H}$, $u=(w,x_j)$ and  $v=(w, x_k)$, with $w\in (V_H)^n$  and $x_j\sim x_k$ in $G$. Let $z=(y, x_i)$ be an arbitrary vertex of $V_{G\wr H}$. Then we have
 $$d_{G\wr H}(z, u)> d_{G\wr H}(z,v)\iff \rho_{\delta(w,y)}(x_i,x_j)>\rho_{\delta(w,y)}(x_i,x_k).$$
\end{lemma}
\begin{proof}
By virtue of Theorem \ref{dista}, we have:
$$
d_{G\wr H}(z, u)= d_{H^{\Box n}}(y,w)+\rho_{\delta(w,y)}(x_i,x_j) \quad \textrm{ and } \ \ d_{G\wr H}(z, v)= d_{H^{\Box n}}(y,w)+\rho_{\delta(w,y)}(x_i,x_k).
$$
The claim follows.
\end{proof}

\begin{lemma}\label{n2e}
Consider two vertices $u,v\in V_{G\wr H}$, $u=(w,x_j)$ and  $v=(w, x_k)$, with $w\in (V_H)^n$  and $x_j\sim x_k$ in $G$. Let us put $E=\{u,v\}\in E_{G\wr H}$ and $e=\{x_j,x_k\}\in E_G$. Then we have:
$$
n_u(E)= \sum_{A\subseteq V_G}{(m-1)^{|A|} n_{x_j}(e,\rho_A)}.
$$
\end{lemma}
\begin{proof}
By Lemma \ref{n2}, a vertex $z=(y,x_i)\in V_{G\wr H}$ is in $B_u(E)$ if and only if $x_i$ is in $B_{x_j}(e, \rho_{\delta(w,y)})$.
For a fixed $x_i\in V_G$, and a fixed $A\subseteq V_G$, there are exactly $(m-1)^{|A|}$ vertices $z=(y, x_i)\in V_{G\wr H}$ such that $A=\rho_{\delta(w,y)}$. Then:
\begin{equation*}\begin{split}
n_u(E)&=|\{z=(y,x_i)\in V_{G\wr H}: \rho_{\delta(w,y)}(x_i,x_j)<\rho_{\delta(w,y)}(x_i,x_k)\}|\\&=
\sum_{A\subseteq V_G} |\{z=(y,x_i)\in V_{G\wr H}: \delta(w,y)=A,\, \rho_{A}(x_i,x_j)<\rho_{A}(x_i,x_k)\}|\\&=
\sum_{A\subseteq V_G} |\{y \in (V_{H})^n: \delta(w,y)=A \}| |\{x_i\in V_G:  \rho_{A}(x_i,x_j)<\rho_{A}(x_i,x_k)\}|\\&=
\sum_{A\subseteq V_G} (m-1)^{|A|} n_{x_j}(e,\rho_A).
\end{split}
\end{equation*}
\end{proof}

\begin{remark}\label{inva2}
The quantities $n_u(E)$ and $n_v(E)$, associated with an edge $E$ of type II, do not depend on the configuration of the lamps:
consider $E=\{u,v\}$, with $u=(w,x_j),v=(w, x_k)\in V_{G\wr H}$, and $E'=\{u',v'\}$, with $u'=(w',x_j),v'=(w', x_k)\in V_{G\wr H}$, then
$n_u(E)=n_{u'}(E)$ and $n_v(E)=n_{v'}(E)$.
\end{remark}

\begin{defi}\label{dd}
Let $G=(V_G,E_G)$ be a connected graph, and let $A$ and $B$ be two subsets of $V_G$. We put
$$
Sz(G,A,B):=\sum_{e=\{x_j,x_k\}\in E_G} n_{x_j}(e,\rho_A) n_{x_k}(e,\rho_B).
$$
\end{defi}

\begin{prop}\label{n2p}
Let $G$ and $H$ be connected graphs with $|V_G|=n$, and $|V_H|=m$. Then
$$
Sz_{II}(G\wr H)=m^n \sum_{A,B\subseteq V_G}(m-1)^{|A|+|B|} Sz(G,A,B).
$$
\end{prop}
\begin{proof}
By definition, we have
$$
Sz_{II}(G\wr H)=\sum_{E=\{u,v\}\in E_{II}} n_u(E)n_v(E).
$$
We fix $w\in  (V_H)^n$ and set $u_i=(w,x_i)\in V_{G\wr H}$, for every $x_i\in V_G$, and $E_{j,k}=\{u_j,u_k\}\in E_{G\wr H}$, for every $\{x_j,x_k\}\in E_G$.
By virtue of Remark \ref{inva2}, we have:
$$\sum_{E=\{u,v\}\in E_{II}} n_u(E)n_v(E)= m^n \sum_{e=\{x_j,x_k\}\in E_G} n_{u_j}(E_{j,k} ) n_{u_k}(E_{j,k}).$$
Finally, by combining with Lemma \ref{n2e}, we obtain
\begin{equation*}
\begin{split}
Sz_{II}(G\wr H)&= m^n \sum_{e=\{x_j,x_k\}\in E_G} {\sum_{A\subseteq V_G}{(m-1)^{|A|} n_{x_j}(e,\rho_A)} \sum_{B\subseteq V_G}{(m-1)^{|B|} n_{x_k}(e,\rho_B)}}\\&=m^n \sum_{A,B\subseteq V_G}(m-1)^{|A|+|B|} Sz(G,A,B).
\end{split}
\end{equation*}
\end{proof}
By gluing together Proposition \ref{n1p} and Proposition \ref{n2p}, and using Equation \eqref{split}, we obtain the following theorem.

\begin{theorem}\label{szegen}
Let $G=(V_G,E_G)$ and $H=(V_H,E_H)$ be connected graphs with $|V_G|=n$, and $|V_H|=m$. Then we have
$$
Sz(G\wr H)=n^3 m^{3n-3} Sz(H)+m^n \sum_{A,B\subseteq V_G}(m-1)^{|A|+|B|}Sz(G,A,B).
$$
\end{theorem}

\begin{remark}
Let $H_1$ and $H_2$ be connected graphs with $|V_{H_1}|=|V_{H_2}|$ and $Sz(H_1)=Sz(H_2)$. Then we have:
$$
Sz(G\wr H_1)=Sz(G\wr H_2), \qquad \textrm{for every graph }G.
$$
\end{remark}

We can exploit possible symmetries in order to simplify the formula of Theorem \ref{szegen} for some special classes of graphs.
\begin{lemma}\label{lemminojan}
Let $e,e'\in E_G$, where $e=\{u,v\}$, $e'=\{u',v'\}$, with $u,u',v,v'\in V_G$. Suppose that there exists $\phi\in Aut(G)$ such that $\phi(u)=u'$ and $\phi(v)=v'$. Then
$$
n_u(e)=n_{u'}(e'),\;\;n_v(e)=n_{v'}(e').
$$
\end{lemma}
\begin{proof}
We have
\begin{equation*}\begin{split}
\phi (\{w\in V_G: d_G(w,u)<d_G(w,v) \})&=
 \{\phi(w)\in V_G: d_G(w,u)<d_G(w,v) \}\\&=
  \{z \in V_G: d_G(\phi^{-1}(z),u)<d_G(\phi^{-1}(z),v) \}\\&=
    \{z \in V_G: d_G(z,u')<d_G(z,v') \}
    \end{split}
    \end{equation*}
  and then $n_u(e)=n_{u'}(e')$. By exchanging the role of $u$ and $v$, we get $n_v(e)=n_{v'}(e').$
\end{proof}
\begin{remark}\label{simmetrie}
It follows from Lemma \ref{lemminojan} that, if a graph $G=(V_G,E_G)$ is edge-transitive, then $Sz(G) =|E_G| n_u(e)n_v(e)$ for any $e=\{u,v\}\in E_G$. In particular, if $G$ is also arc-transitive, $Sz(G) =|E_G| n_u(e)^2$, for any $e=\{u,v\}\in E_G$. Therefore, if $G$ and $H$ are graphs with $|V_G|=n$, $|V_H|=m$, and $G$ is edge-transitive, by Remark \ref{inva2} we have:
$$
Sz_{II}(G\wr H)=|E_{II}| n_u(E)n_v(E)=m^n |E_G|n_u(E)n_v(E),
$$
for any $E=\{u,v\}\in E_{II}$. If $G$ is also arc-transitive,
$$
Sz_{II}(G\wr H)=m^n |E_G|n_u(E)^2.
$$
\end{remark}
As an application of Remark \ref{simmetrie}, in the following  theorem we explicitly compute the Szeged index $Sz(K_n\wr H)$, where $H$ is any connected graph.
\begin{theorem}\label{Tkn}
Let $K_n$ be the complete graph on $n$ vertices and let $H$ be a connected graph with $m$ vertices. Then:
$$
Sz(K_n\wr H)=n^3 m^{3n-3} Sz(H)+\frac{1}{2}m^n n (n-1)\left(m+m^{n-2}(m^2+mn-3m-n+2)\right)^2.
$$
\end{theorem}
\begin{proof}
We fix $e=\{x_j,x_k\}\in E_{K_n}$. For any subset $A\subseteq V_{K_n}$, with $|A|\geq 2$, looking at Equation \eqref{rocompleto} of Example \ref{K_n}, it is easy to see that:
$$
B_{x_j}(e,\rho_A)=
\begin{cases}
\emptyset \; &\mbox{if } x_j,x_k\notin A\\
V_G\setminus \{x_j\}\;  &\mbox{if } x_j\in A, x_k\notin A\\
\emptyset \; &\mbox{if } x_j\notin A, x_k\in A\\
\{x_k\}\;  &\mbox{if } x_j,x_k\in A.
\end{cases}
\qquad
n_{x_j}(e,\rho_A)=
\begin{cases}
0\; &\mbox{if } x_j,x_k\notin A\\
n-1\;  &\mbox{if } x_j\in A, x_k\notin A\\
0\; &\mbox{if } x_j\notin A, x_k\in A\\
1\;  &\mbox{if } x_j,x_k\in A.
\end{cases}
$$
If $A=\emptyset$, the set $B_{x_j}(e,\rho_{\emptyset})$ contains only the vertex $x_j$, so that $n_{x_j}(e,\rho_{\emptyset})=1$. If $A=\{x\}$, then $n_{x_j}(e,\rho_A)$ is non-zero only if $x=x_j$: if this is the case, one has $n_{x_j}(e,\rho_{\{x_j\}})=n$.
Notice that
$$
|\{A\subseteq V_G: |A|=h,   x_j\in A, x_k\notin A  \}|= {n-2\choose h-1}, \quad  \textrm{ for each } h=2,\ldots,n-1;
$$
$$
|\{A\subseteq V_G: |A|=h,   x_j\in A, x_k\in A  \}|= {n-2\choose h-2}, \quad \textrm{ for each } h=2,\ldots, n.
$$
By an explicit computation we have:
\begin{equation}\label{conto}
 \begin{split}
 &\sum_{A\subseteq V_G}{(m-1)^{|A|} n_{x_j}(e,\rho_A)}\\&=1+n(m-1)+\sum_{h=2}^{n-1} {n-2\choose h-1} (m-1)^h (n-1)+ \sum_{h=2}^{n} {n-2\choose h-2} (m-1)^h  \\
 %&=1+n(m-1)-(n-1)(m-1) + \sum_{i=0}^{n-2} {n-2\choose i} (m-1)^{i+1} (n-1)+\sum_{i=0}^{n-2} {n-2\choose i} (m-1)^{i+2}\\
 &= m+(n-1)(m-1)m^{n-2}+ (m-1)^2m^{n-2}.
 \end{split}
\end{equation}
Since $K_n$ is arc-transitive, by virtue of Remark \ref{simmetrie}, we have $Sz_{II}(K_n\wr H)=m^n |E_{K_n}| n_u(E)^2$,
for any edge $E$ of type II in $K_n\wr H$. Combining Lemma \ref{n2e} with Equation \eqref{conto}, we obtain:
$$
Sz_{II}(K_n\wr H)=\frac{1}{2}m^n n (n-1)\left(m+(n-1)(m-1)m^{n-2}+ (m-1)^2m^{n-2}\right)^2.
$$
Finally, combining with Proposition \ref{n1p}, we get the thesis.
\end{proof}

\begin{example}\label{cappa24}
Consider the wreath product $K_2\wr C_4$ (see Fig. \ref{figK2C4}). According with our notation, we have $n=2$ and $m=4$. Moreover, as the graph $K_2$ is $1$-regular and $C_4$ is $2$-regular, the wreath product $K_2 \wr C_4$ is a $3$-regular graph on $32$ vertices with $48$ edges. The partition between the $32$ edges of type I and the $16$ edges of type II is also graphical in Fig. \ref{figK2C4}: the horizontal and vertical edges are edges of type I (representing a change of configuration, that is, a step in a copy of $C_4$), the diagonal ones are of type II (representing a change of position, that is, a step in a copy of $K_2$). There is no automorphism of $K_2\wr C_4$ sending an edge of type I to an edge of type II or viceversa (since an edge of $K_2\wr C_4$  is of type I if and only if it is contained in a $4$-cycle, but any automorphism maps $4$-cycles to $4$-cycles). In particular, $K_2\wr C_4$ is not edge-transitive. However, by a direct computation, $n_u(e)=n_v(e)=16$ for any edge $e=\{u,v\}\in E_{K_2\wr C_4}.$ This gives
$$
Sz(K_2\wr C_4) = \sum_{e=\{u,v\}\in E_{K_2\wr C_4}}n_u(e)n_v(e) =16^2 |E_{K_2\wr C_4}| =12288.
$$
Observe that this is also the value expected from Theorem \ref{Tkn}, since, as we have seen in Example \ref{szC_m}, one has $Sz(C_4)=16$.
\begin{figure}[h]
\begin{center}
\includegraphics[width=0.37\textwidth]{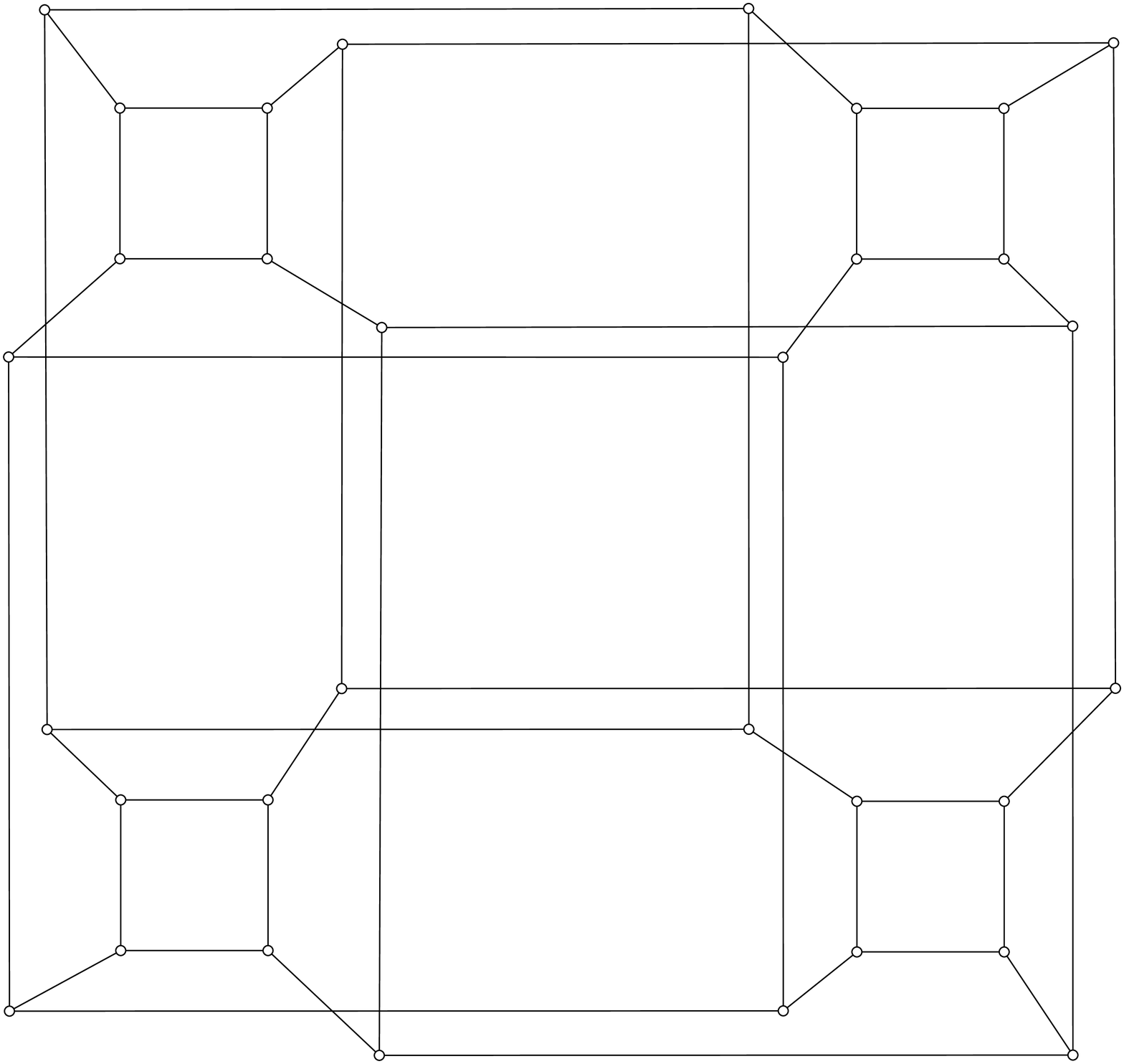}
\end{center}\caption{The graph $K_2 \wr C_4$.}\label{figK2C4}
\end{figure}
\end{example}

%%%%%%%%%%%%%%%%%%%%%%%%%%%%%%%%%%%%%%%%%%%%%%%%%%%%%%%%%%%%%%%%%%%%%%%%%%%%%%%%%%%%

\end{document}